\documentclass[10pt]{article}
\usepackage[all]{xy}
\usepackage{amsfonts,amsmath,oldgerm,amssymb,amscd}

\usepackage{tikz-cd}
\tikzset{
  symbol/.style={
    draw=none,
    every to/.append style={
      edge node={node [sloped, allow upside down, auto=false]{$#1$}}}
  }}
\usepackage{relsize}
\tikzset{nodes={inner sep=2pt}}
\usepackage{blindtext}
\usetikzlibrary{positioning}
\usepackage{bookmark}

\newtheorem{theorem}{Theorem}[section]
\newtheorem{proposition}[theorem]{Proposition}
\newtheorem{lemma}[theorem]{Lemma}
\newtheorem{definition}[theorem]{Definition}
\newtheorem{corollary}[theorem]{Corollary}

\newcommand{\gj}{\blacksquare}	\newcommand{\gl}{\lambda}
\newcommand{\gm}{\gamma}	
		
\newcommand{\gt}{\theta}

	\newcommand{\BN}{\mbox{$\mathbb N$}}
	
	\newcommand{\BR}{\mbox{$\mathbb R$}}

	\newcommand{\BZ}{\mbox{$\mathbb Z$}}

\newcommand{\CC}{\mbox{$\mathcal C$}}

\newcommand{\CI}{\mbox{$\mathcal I$}}

	\newcommand{\p}{\mbox{$\mathfrak p$}}

	\newcommand{\op}{\mbox{$\oplus$}}
\newcommand{\Spec}{\mbox{\rm Spec\,}}

\newcommand {\Pic}{{\rm Pic }}
\newcommand {\Nil}{{\rm Nil }}
\newcommand {\IIm}{{\rm Im }}

\newcommand {\coker}{{\rm coker }}
\newcommand {\Id}{{\rm Id }}
\newcommand {\U}{{\rm U}}
\newcommand {\cl}{{\rm cl}}
\newcommand {\gp}{{\rm gp}}
\newcommand {\QF}{{\rm QF}}
\newcommand {\PQF}{{\rm PQF}}
\newcommand {\Rad}{{\rm Rad}}

\newcommand{\tens}[1]{%
  \mathbin{\mathop{\otimes}\limits_{#1}}%
}

\begin{document}

%\tableofcontents

\begin{center}
 {\Large \bf Subintegrality and ideal class groups of monoid algebras}\\\vspace{.2in}
 {\large Md Abu Raihan,  Leslie G. Roberts 
 and Husney Parvez Sarwar }\\\vspace{.1in}
\end{center}
\begin{abstract}
$(1)$ Let $M\subset N$ be a commutative cancellative torsion-free and subintegral extension of monoids. Then we prove
that in the case of ring extension $A[M]\subset A[N]$, the two notions, subintegral and weakly subintegral coincide provided $\mathbb{Z}\subset A$.
$(2)$ Let $A \subset B$ be an extension of commutative rings and $M\subset N$ an extension of commutative cancellative torsion-free positive monoids. Let $I$ be a radical ideal in $N$. Then  $\frac{A[M]}{(I\cap M)A[M]}$ is subintegrally closed in $\frac{B[N]}{IB[N]}$ if and only if the group of invertible $A$-submodules of $B$ is isomorphic to the group of invertible $\frac{A[M]}{(I\cap M)A[M]}$-submodules of  $\frac{B[N]}{IB[N]}$.
\end{abstract}

%(Milnor Square)

\tableofcontents

\section{Introduction}
{\it Unless otherwise stated, all rings are commutative with the multiplicative identity $1$ and also all monoids are commutative cancellative torsion-free with the identity $1$.}
% Monoids appear naturally in the affine toric variety 
Monoids appear naturally in the theory
of affine toric varieties where one considers an affine cone in a $d$-dimensional real vector space and then intersects with $d$-dimensional lattice. More generally, one considers a monoid $M$ generated by monomials in the free monoid $\mathbb{Z}^n\oplus \mathbb{Z}^m_+$ (in the additive notation). We are interested in studying how far the K-theoretical properties of the Laurent-polynomial algebra $A[\mathbb{Z}^n\oplus \mathbb{Z}^m_+]$ can be extended to the monoid algebra $A[M]$, where $A$ is a ring. It was conceived by Anderson \cite{An78,An88} and then successfully developed by Gubeladze in a series of excellent papers \cite{Gu89,Gu05,Gu18}. In this topic, using CDH techniques, Cortin\~as--Haesemeyer--Walker--Weibel \cite{CHMWW2012,CHMWW2018} resolved a conjecture of Gubeladze in characteristic $p>0$ and in the mixed characteristics. An old problem in the $K$-theory of monoid algebras is solved by Krishna--Sarwar \cite{KrSa17} using pro-descent and pro-calculations. Following Anderson and Gubeladze, the third author had studied a similar kind of problems
on $K$-theory of monoid algebras in \cite{Sa16,Sa21} and in \cite{KeSa15} with Keshari. The second author with Reid studied 
subintegrality and weak subintegrality of some monomial subrings 
of $\mathbb{Z}[t_1,\ldots,t_n]$ in \cite{RR00,RR01}.
%These monomial subrings can be thought of as a monoid $\mathbb{Z}$-algebra with coefficients (see section $5$).
In the literature, many algebro-geometric properties of monoid
algebras have been studied in the name of semigroup rings.
For a comprehensive idea about the same, we refer the reader to two beautiful books \cite{CLS2011},\cite{B-G}.

In this paper also, we study some $K$-theoretical properties of an extension of monoid algebras. Subintegrality is intricately linked with the Picard group and invertible modules. To demonstrate this, let us recall a result of Dayton.  Let $A$ be a reduced graded algebra over a field $k$ and $B$ its seminormalization (for the definition, see Definition \ref{Weakly subintegral extensions}$(\ref{RingExtn:Item3})$). 
Then Dayton \cite{Dayton} proved that there is a functorial map $\Pic(A)\longrightarrow B/A$, which is an isomorphism if char(k)=0.
A generalization of the above result is provided by Roberts and Singh \cite{Roberts and Singh} for a subintegral extension of commutative rings (for the definition, see Definition \ref{Weakly subintegral extensions}$(\ref{RingExtn:Item3})$).
More precisely, they prove the following. Let   $A\subset B$ be the subintegral extension of commutative rings such that $\mathbb{Q}\subset A$. Then there is a functorial morphism $B/A\longrightarrow \mathcal{I}(A,B)$, which is an isomorphism,
where $\mathcal{I}(A,B)$ is the group of invertible $A$-submodules of $B$. 
%In characteristic $p>0$ the concept of subintegrality splits into two concepts subintegrality and weak subintegrality. 
Subintegrality often splits into two concepts, subintegrality
and weak subintegrality. 
The latter was defined by Yanagihara \cite{Yanagihara1985} (for the definition, see Definition (\ref{Weakly subintegral extensions})) (see also \cite{Reid Roberts Singhweak}). It turns out that both the definitions, subintegral extension and weak subintegral extension, coincide for an extension of $\mathbb{Q}$-algebras. 
In section \ref{On monoid subintegral extension}, we prove the following result which says in the case of monoid extension, the two notions coincide.

\begin{theorem}(Theorem $\ref{twkly1}$)
Let $M\subset N$ be an extension of monoids and $A$ a ring.
Then the following are equivalent if $\BZ \subset A $.

$(1)$ The extension $A[M]\subset A[N]$ is subintegral.

$(2)$ The extension $A[M]\subset A[N]$ is  weakly subintegral.

$(3)$ The monoid extension $M\subset N$ is subintegral.
\end{theorem}

The above theorem does not hold true if $A$ is an $\mathbb{F}_{p}$-algebras (see Example \ref{counter-example}). If
 $N$ is not a monoid, then also the above result is not true even though $A$ is a $\mathbb{Z}$-algebras (see Example \ref{counter-z}).

In another direction, Swan \cite{Swan1980} proved the homotopy invariance property for the Picard group. More precisely, he proved:  $A$ is a seminormal ring if and only if $\Pic(A)\cong \Pic(A[X])$.
Motivated by the result of Swan \cite{Swan1980}, Sadhu and Singh \cite{Sadhu and Singh} proved that $A$ is subintegrally closed in $B$ if and only if
$\mathcal{I}(A, B)\cong \mathcal{I}(A[X], B[X])$. Motivated by the results of Sadhu--Singh \cite{Sadhu and Singh} and Anderson \cite{An88},
Sarwar \cite{Sarwar} proved the following result. 

  Let $A\subset B$ be an extension of rings and $M \subset N$ an extension of positive monoids.
  Assume $N$ is affine (for definition, see Definition \ref{sec-2:monoid}(\ref{Monoid-Item5})).
 Then $A[M]$ is subintegrally closed in $B[N]$  if and only if $\mathcal{I}(A, B) \cong \mathcal{I}(A[M], B[N])$.

We prove the following result, and if we set $I = \emptyset $ in the monoid $N$
so that $IB[N] = 0$, then it becomes \cite[Theorem 1.2]{Sarwar} (see Theorem \ref{Sarwar-Main-Thm}). Thus it is a generalization of \cite[Theorem 1.2]{Sarwar}.

\begin{theorem}(Theorem $\ref{MainTheorem-1}$)
\label{Main-Thm-Intro}
Let $A\subset B$ be an extension of rings and $M\subset N$ an extension of positive monoids. Let $I$ be a radical ideal in $N$.

  \begin{enumerate}
  \item[(a)]
   If $\frac{A[M]}{(I\cap M)A[M]}$ is subintegrally closed in $\frac{B[N]}{IB[N]}$ and $N$ is affine, then $\mathcal{I}(A, B)\cong \mathcal{I}\left(\frac{A[M]}{(I\cap M)A[M]}, \frac{B[N]}{IB[N]}\right)$.
   
   \item[(b)]
   If $B$ is reduced, $A$ is subintegrally closed in $B$, and $M$ is subintegrally closed in $N$, then $\mathcal{I}(A, B)\cong \mathcal{I}\left(\frac{A[M]}{(I\cap M)A[M]}, \frac{B[N]}{IB[N]}\right)$.

   \item[(c)]
   If $M=N$, and $A$ is subintegrally closed in $B$, then $\mathcal{I}(A, B)\cong \mathcal{I}\left(\frac{A[M]}{IA[M]}, \frac{B[M]}{IB[M]}\right)$.

   \item[(d)]
   If $\mathcal{I}(A, B)\cong \mathcal{I}\left(\frac{A[M]}{(I\cap M)A[M]}, \frac{B[N]}{IB[N]}\right)$ and $N$ is affine, then (i) $A$ is subintegrally closed in $B$, (ii) $\frac{A[M]}{(I\cap M)A[M]}$ is subintegrally closed in $\frac{B[N]}{IB[N]}$, (iii) $B$ is reduced or $M=N$.
\end{enumerate}
\end{theorem}

%Sadhu and Singh \cite[Theorem 2.6]{Sadhu and Singh} discussed the relationship between $\CI(A,B)$ and $\CI(A[X],B[X])$, when $A$ is not subintegrally closed in $B$. Sarwar \cite[Theorem 1.4]{Sarwar} generalized their result to the monoid algebra. In Theorem \ref{diagram-thm}, we have discussed the similar result for the generalized discrete Hodge algebra.
 We also consider when the ring extensions is not subintegrally closed and prove the following result.
\begin{theorem}(Theorem $\ref{diagram-thm}$)
 Let $A\subset B$ be an extension of rings and let $^+\!\!A$ denote the subintegral closure of $A$ in $B$. Assume $M$ is an affine positive monoid, and $I$ is a radical ideal in $M$. Then\\
 (i) the following diagram 
 
 \[
% {\large
\begin{tikzcd}
1 \arrow[r] & \CI(A,^+\!\!A) \arrow[r, "i"] \arrow[d, "{\theta(A,^+\!\!A)}"'] & \CI(A,B) \arrow[r, "{\phi(A,^+\!\!A,B)}"] \arrow[d, "{\theta(A,B)}"'] & \CI(^+\!\!A,B) \arrow[r] \arrow[d, "{       \theta(^+\!\!A,B)}" '] & 1 \\
1 \arrow[r] & \CI\left(\frac{A[M]}{IA[M]},\frac{^+\!\!A[M]}{I^+\!\!A[M]}\right) \arrow[r]                                    & \CI\left(\frac{A[M]}{IA[M]},\frac{B[M]}{IB[M]}\right) \arrow[r, "\phi'"']                                      & \CI\left(\frac{^+\!\!A[M]}{I^+\!\!A[M]},\frac{B[M]}{IB[M]}\right) \arrow[r]                               & 1
\end{tikzcd}
%}
\]
is commutative with exact rows, where $\phi' = \phi\left(\frac{A[M]}{IA[M]},\frac{^+\!\!A[M]}{I^+\!\!A[M]},\frac{B[M]}{IB[M]}\right)$.\\
(ii) If $\mathbb{Q}\subset A$, then $\CI\left(\frac{A[M]}{IA[M]},\frac{^+\!\!A[M]}{I^+\!\!A[M]}\right)\cong \frac{\mathbb{Z}[M]}{I\mathbb{Z}[M]}\otimes_{\mathbb{Z}}\CI(A,^+\!\!A).$
 \end{theorem}
 
%%%%%%%%%%%%%%%%%%%%%%%%%%%%%%%%%%%%%%%%%%%%%%%%%
\par
The structure of the paper is as follows. In Section~\ref{Sec_preli}, we review the foundational material on monoids and related ring-theoretic concepts, together with the basic results used throughout the paper. Section~\ref{On monoid subintegral extension} is devoted to studying the equivalence of the notions of subintegrality and weak subintegrality for an extension of monoid algebras (see Theorem~\ref{twkly1}). Before proving Theorem~\ref{twkly1}, we prove a crucial combinatorial lemma (see Lemma \ref{Lemma3}) which is used in Theorem~\ref{twkly1}.  In Section~\ref{Sec_4}, we first recall the definition of a Milnor square in the category of rings and then extend this notion to the category of ring extensions. We derive a crucial exact sequence of abelian groups (see Lemma~\ref{7l1}) associated with such a Milnor square.
\par
Further, we establish a key result (Lemma~\ref{Lm2}) on the subintegral closed-ness of an extension of quotient monoid algebras by radical ideals. We also compare the units of positive generalized discrete Hodge algebras with those of the base ring (Lemma~\ref{LmUnits}). Several results involving Cartesian squares under the $\Nil$ functor are proved (Lemmas~\ref{LmFibProdNill_1} and \ref{LmFibProdNill_2}), which lead to important isomorphisms as modules (Theorem~\ref{Nil-Thm}). These are central tools in the proof of our main result, Theorem~\ref{MainTheorem-1}.
\par
Finally, we address the case when the ring extension is not subintegrally closed and compute the corresponding ideal class group over monoid algebras quotient by radical ideals (see Theorem~\ref{diagram-thm}).

%%%%%%%%%%%%%%%%%%%%%%%%%%%%%%%%%%%%%%%%%%%%%%%%%%%%%%%

\section{Basic definitions and results}\label{Sec_preli}
In Section \ref{sec:monoid} we review several concepts about monoids. In Section \ref{Weakly subintegral extensions} we discuss similar ring theoretic concepts. In many cases a monoid algebra $A[M]$ satisfies a certain ring theoretic property if and only if $M$ satisfies the corresponding monoid theoretic property. However Theorem \ref{twkly1}  suggests that there is no monoid theoretic concept of weak subintegrality corresponding in this way to weak subintegrality in ring theory.

%In Section \ref{sec:monoid}, we review several concepts about monoids. In Section \ref{Weakly subintegral extensions}, we discuss similar ring theoretic concepts. From the below, it will be clear that for every notion in ring theory, there is a similar notion in monoid theory. However, Theorem \ref{twkly1} suggests that there is no such notion of weak subintegrality in monoid theory. 

\subsection{Monoids}\label{sec:monoid}
In this subsection, first, we recall
the definition of commutative, cancellative and torsion-free monoid.
Then we recollect some basic notions pertaining to monoids from \cite[Chapter 2]{B-G}. 

%However, in our results, the quotient monoid appears naturally. To define this, we need pointed monoids. Therefore we recall basic notions of pointed monoids in the Definition \ref{PointedMonoid} from \cite[Section 1]{CHMWW2015}.

\begin{definition}\label{sec-2:monoid}

 % Monoids, their morphism, their ideal, and the concept of subintegral extensions for ctf monoids are all going to be discussed here. For basic results pertaining to monoids, we refer the reader to  \cite[Ch-2]{B-G}.

\begin{enumerate}
    \item \label{Monoid-Item1}
    A commutative monoid is a non-empty set $M$ with a binary operation $\mu: M \times M \rightarrow M$ ( written $\mu(x,y) =x\cdot y$ ) which is commutative (i.e., $x\cdot y = y\cdot x$), associative  (i.e., $(x\cdot y) \cdot z =  x\cdot (y \cdot z)$ ), and has the neutral element $1$ (i.e., $x\cdot1 =x$). We will usually call the operation multiplication and we sometimes write  $x \cdot y=x y$ ($\forall x, y \in M$) if there is no confusion.
   
    \item\label{Monoid-Item2}
    A function $f:M\longrightarrow N$ is a homomorphism of  monoids if $f(xy) = f(x)f(y)$, $\forall \ x, y\in M$, and $f(1) = 1$ . We say $M$ and $N$ are isomorphic if there exists another monoid homomorphism $g:N\rightarrow M$ such that $gf=I_M, fg=I_N$ where $I_{M}$ and $I_{N}$ denotes the identity monoid morphism on $M$ and $N$ respectively.
    
        \item\label{Monoid-Item3}
    Let $\gp(M)$ denote the group of fractions of the monoid $M$.
    For the construction and uniqueness of $\gp(M)$, we refer the reader to \cite[page 50]{B-G}.
  
    % For a monoid $M$, we define another monoid $M_{+} = M\cup {\{0\}}$ with the definition $x\cdot0 = 0$, $\forall \ x\in M$. The element $'0'$ is called the base point of $M$. This monoid is called the pointed monoid and sometime we say $M_{+}$ is the augmented monoid. The morphism of pointed monoids are homomorphism of monoids which takes the base point to the base point.

      \item\label{Monoid-Item4}
    A monoid $M$ is called 
   %commutative if $xy=yx$, $\forall \ x, y \in M$, 
   cancellative if $xy=zy$, implies $x=z$ $\forall \ x, y, z\in M$ and torsion-free if for $ x, y \in M$, whenever $x^n=y^n$ implies $x=y$ for $n\in \BN$, the set of all positive integers. 
   A monoid $M$ is cancellative if and only if the map $M\longrightarrow \gp(M)$ is injective  (see Section $2.A$  of \cite{B-G}).  It is easy to check that a finitely generated cancellative monoid $M$ is torsion-free if and only if $\gp(M)$ is torsion-free (see Section $2.A$  of \cite{B-G}).

     \item\label{Monoid-Item5}
  For a monoid $M$, let $\U(M)$ denote the group of units (i.e., invertible elements) of $M$. If $\U(M)$ is a trivial group, i.e., $\U(M)=\{1\}$, then $M$ is called positive. 
  A monoid is affine if it is finitely generated and isomorphic to a submonoid of a free abelian group $\mathbb{Z}^d$ for some $d\geq 0$.
  Within the class of commutative monoids, the affine monoids are characterized by being finitely generated, cancellative and torsion-free \cite[page $49$]{B-G}.

    \item\label{Monoid-Item6}
  An ideal of a monoid $M$ is a subset $I \subset M$ such that $I M \subset I$ $(I M:=\{x m \mid x \in I, m \in M\})$. 
  If $M\subset N$ is an extension of monoids and  $I$ is an ideal of $N$, then clearly, $I\cap M$ is also an ideal $M$.  A proper ideal $\mathfrak{p}$ is a prime ideal of $M$ if  $xy \in \mathfrak{p} \Rightarrow x \in \mathfrak{p}$ or $y \in \mathfrak{p}$ for all $x,y \in M$.   Let $I$ be an ideal of $M$. Then $\Rad(I):=\left\{x \in M \mid x^{n} \in I\right.$ for some $\left.n \in \mathbb{N}\right\}$. An ideal $I$ is called a radical ideal if $I=\Rad(I)$.
  
     \item\label{Monoid-Item7}
  Let $M\subset N$ be an extension of cancellative torsion free monoids. Let  $xM=\{xm: m\in M\}$. The extension $M\subset N$ is called elementary subintegral if $N = M\cup xM$ for some $x\in N$ with $x^{2}$, $x^{3}\in M$ (note in this case, $N$ is a monoid). If $N$ is a union of submonoids which are obtained from $M$ by a finite succession of elementary subintegral extensions, then the extension $M\subset N$ is called subintegral. The subintegral closure of $M$ in $N$, denoted by $M_{N}^s$ or simply $M^s$, is the largest subintegral extension of $M$ in $N$. We say $M$ is subintegrally closed in $N$ if $M_{N}^s = M$ i.e., for $x\in N$ whenever $x^2,x^3 \in M \Rightarrow x\in M$. $M$ is seminormal if it is subintegrally closed in $\gp(M)$.

   \item \label{Monoid-Item7B}
    Let $M$ be a submonoid of a commutative monoid $N$. The saturation or integral closure of M in N is the submonoid 
    $$\widetilde{M}_N:=\{  x\in N \mid x^n\in M \ \textit{for same}\ n\in \BN \}$$
    of $N$. One calls $M$ saturated or integrally closed in $N$ if $\widetilde{M}_N=M$. The normalization $\bar{M}$ of a cancellative monoid $M$ is the integral closure of $M$ in $\gp(M)$, and if $\bar{M}=M$, then M is called normal. Observe that normal monoid is a seminormal monoid.

 %  \item\label{Monoid-Item8}
 %   A monoid $M$ is reduced if whenever $x^2=y^2$ and $x^3=y^3$ for some $x,y\in M$, then $x=y$. Equivalently, $M$ is reduced if and only if whenever $x^n=y^n$, $\forall \ n>>0$ implies $x=y$.

  %  \item \label{Monoid-Item9}
  %   \textbf{(Localization)} Let $M$ be a monoid with the multiplicative identity $1$. A subset $S\subset M$ is called a multiplicatively closed set if $1\in S$ and $xy\in S$ for all $x,y\in S$.  Let $M$ be a monoid and $S \subset M$ a multiplicatively closed subset. Define $S^{-1}M$ to be the monoid with elements $x/s$, $x\in M$  and $s\in S$, where $x/s = y/t$ if there is a $u\in S$ such that $u(xt) = u(ys)$. The multiplication in $S^{-1}M$ is induced by $M$, $(x/s)(y/t) = xy/st$. Note that $(1/s)(s/1) = 1$ so that every element of $S$ becomes a unit in $S^{-1}M$. The monoid $S^{-1}M$ is called the monoid of fractions of $M$ with respect to $S$ or the localization of $M$ at $S$.

\end{enumerate}
\end{definition}

\subsection{Ring extensions and monoid algebras}
\begin{definition}\label{Weakly subintegral extensions}

\begin{enumerate}
   
   \item \label{RingExtn:Item1}
     For a ring $A$, let $\U(A)$ denote the group of units of $A$ and $\Pic(A)$ denote the Picard group of rank one projective $A$-modules where the group operation is given by the tensor product of modules.

       For a ring homomorphism $f: A\longrightarrow B$, we have a morphism of abelian groups $\U(f) : \U(A)\longrightarrow \U(B)$ defined as $\U(f)(a) = f(a)$ and $\Pic(f): \Pic(A) \rightarrow \Pic(B)$ defined as $\Pic(f)([P]):=[P\otimes_{A}B]$, where $[P]\in \Pic(A)$. It is easily checked that $\U$, $\Pic$ are  functors from the category of rings to the category of abelian groups.

     The nilradical of a ring $A$, which is denoted by $\Nil(A)$, is the radical ideal of the zero ideal, i.e., $\Nil(A) :=\Rad(0) = \{x\in A : x^n = 0$  for some   $n\geq 1 \}$.

  \item  \label{RingExtn:Item2}
  Let $A\subset B$ be an extension of rings. The set of all $A$-submodules of $B$ is a commutative monoid under multiplication, with the identity $A$. Let  $\CI(A, B)$ denote the group of invertible elements of the monoid. The elements of the group  $\CI(A, B)$ are called invertible $A$-submodules of $B$.

    \item \label{RingExtn:Item3}
   An  extension of rings $A\subset B$ is defined to be a subintegral extension if $B$ is integral over $A$, the associated map $\Spec(B)\longrightarrow \Spec(A)$ is a bijection and $Q\in \Spec(B)$, $k(A\cap Q)\cong k(Q)$, where $k(Q) = B_Q/QB_Q$, $k(P)=A_P/PA_P$ for $P\in \Spec(A)$.
   Let $A\subset B$ be an extension of rings. The extension $A\subset B$ is called elementary subintegral if $B = A[b]$ for some $b\in B$ with $b^{2}, b^{3}\in A$. If $B$ is a filtered union of subrings which can be obtained from $A$ by a finite succession of elementary subintegral extensions, then the extension $A\subset B$ is subintegral (for a proof, see Section $2$ of \cite{Swan1980}). The subintegral closure of $A$ in $B$, denoted by $A_{B}^{s}$ or simply $A^s$, is the largest subintegral extension of $A$ in $B$. We say $A$ is subintegrally closed in $B$ if $A_{B}^{s} = A$. It is easy to see that  $A$ is subintegrally closed in $B$ if and only if $b \in B, b^{2}, b^{3} \in A$ imply $b \in A$. A ring A is called seminormal if it is reduced and subintegrally closed in $\PQF(A):= \prod_{\mathfrak{p}}\QF(A/\mathfrak{p})$, where $\p$ runs through the minimal prime ideals of $A$ and $\QF(A/\mathfrak{p})$ is the quotient field of $A/\mathfrak{p}$ (see \cite{B-G}, page $154$). An element $b\in B$ is said to be  subintegral over $A$ if $A\subset A[b]$ is  subintegral extension (see \cite{Singh1996}, page $117$).

     \item \label{RingExtn:Item4}
      ({\bf Weakly subintegral extensions})  In \cite{Yanagihara1985} Yanagihara defined an extension of rings $A \subset B$ to be weakly subintegral if $B$ is integral over $A$, the associated map $\Spec(B) \rightarrow \Spec(A)$ is a bijection, and the associated residue field extensions $k(Q \cap A) \hookrightarrow Q, Q \in \Spec(B)$ are purely inseparable. He defines an extension $A \subset A[b]$ to be elementary weakly subintegral if $b^{p}, p b \in A$ for some prime $p$, and proves that an elementary weakly subintegral extension is weakly subintegral. It is clear from these definitions that the notions of subintegral and weakly subintegral coincide for extensions of $\mathbb{Q}$-algebras, and that a subintegral extension is weakly subintegral. Yanagihara in \cite{Yanagihara1985} proves that given any ring extension $A \subset B$ there is a largest sub-extension ${ }_{B}^{*} A$ such that $A \subset{ }_{B}{ }^{*} A$ is weakly subintegral, and calls ${ }_{B}{ }^{*} A$ the weak normalization of $A$ in $B$. In  \cite[Theorem 2, page 92]{Yanagihara1985} it is proved that ${ }_{B}^{*} A$ is the union of all subrings of $B$ that can be obtained from $A$ by sequence of elementary subintegral or elementary weakly subintegral extensions. An element $b \in B$ is called weakly subintegral over $A$ if $A \subset A[b]$ is a weakly subintegral extension. For more about subintegrality and weak subintegrality, we refer the reader to \cite{Yanagihara1985}.

      \item \label{RingExtn:Item5}
    ({\bf Category of ring extensions}) The category of ring extensions is a category in which objects are ring extension $A\subset B$. A morphism between two ring extensions $A \subset B$ and $C \subset D$ is a ring homomorphism $f: B \longrightarrow D$ which restricts to a ring homomorphism $f_{\mid A}: A \longrightarrow C$ (sometimes we may abuse notation by using $f$ instead of $f_{\mid A}$ ). Compositions are the composition of ring homomorphisms. We shall denote this category by $Ring_{ext}$. \\
   Let $A\subset B$,  $C\subset D$ be ring extensions and $p: (A\subset B)\longrightarrow (C\subset D)$ be a morphism of ring extensions. We define $\CI(p): \CI(A,B)\longrightarrow \CI(C,D)$ as $\CI(p)(M):= Cp(M)$ which is a $C$-module generated by $p(M)$.
  Note that $\mathcal{I}$ is a functor from $Ring_{ext}$ to the category of abelian groups (see \cite[Definition $2.1$]{Roberts and Singh}).
 
       \item  \label{RingExtn:Item6}
    \textbf{(Monoid algebras)}
  Given a monoid $M$ (in multiplicative notation) and a commutative ring $A$ the usual definition of the monoid algebra $A[M]$ is the free $A$-module  with basis $M$. This means that each element $f \in A[M]$ can be represented uniquely as a sum $f = \sum\limits_{x\in M}f_{x}x$, all but a finitely many $f_x \in A$ are zero. Multiplication is defined by $fg = \sum\limits_{x\in M}(\sum\limits_{yz=x}f_{y}g_{z}yz)$, for details see \cite[Chapter $4$, section $4.B$, page-129]{B-G}.

   If $I$ is an ideal of a monoid $M$ then $A[I]$ and $IA[M]$ both denote the ideal generated by $I$ inside $A[M]$.

   % Let $A$ be a commutative ring with $1$. For a pointed set $X$, we define $A[X]$ to mean the free $A$-module on $X$, modulo the summand indexed by the base point of $X$, as stated in \cite[Sec-$1$]{CHMWW2015}  (paragraph after Remark 1.7.1). If $M$ is a pointed monoid, $A[M]$ is a ring in the conventional sense, and multiplication is determined by the product rule for $M$.
    
  %  If $M$ is an unpointed monoid, then \cite[Ch-$4$, Sec- $4.B$]{B-G} defines $A[M]$ to be the free $A$-module generated by elements of $M$ and multiplication in $A[M]$ comes from multiplication in the monoid $M$. If we make the unpointed monoid $M$ into a pointed monoid $M_+$ by adding as base point, then we have $A[M]=A[M_+]$. So we note that the pointed monoid $M_+$ and the unpointed monoid $M$ give the same monoid algebra.
    
   % It's important to note that this definition of monoid algebra and the definition of \cite[Ch-$4$, Sec- $4.B$]{B-G} are compatible. 
 %   If $I$ is an ideal of a pointed or unpointed monoid $M$ then $A[I]$ and $I A[M]$ both denote the ideal generated by $I$ inside $A[M]$. Whether $M$ is pointed or unpointed, it is clear from the definition that $A[M / I]=A[M] / A[I]$ because both sides are the free $A$-module with basis $M \backslash I$.
   % We define $A[M]$ to be the free $A$-module generated by elements of $M$ and multiplication in $A[M]$ comes from multiplication of monoid. 
%  For a given monoid $M$, we note $A[M] = A[M_{+}]$, where $M_{+}$ is the augmented monoid.\\
% For more details about monoid algebras, we refer the reader to \cite[Ch-4]{B-G}.

  \item \label{RingExtn:Item7}
  \textbf{(Generalized discrete Hodge algebras)}
 Let $R$ be a ring, $B:=R[X_1,\ldots,X_n]$ a polynomial ring with $n$-variables and $I$ an ideal of $B$ generated by monomials. Then the ring $A:=B/I$ is called a discrete Hodge algebra. For example, $R[X_1,\ldots,X_n,Y_1, \ldots , Y_m]/(X_iY_j: 1\leq i\leq r; 1\leq j \leq s; r\leq n; s\leq m)$. Let $M$ be a monoid and $I$ an ideal of $M$. Then the ring $A:=R[M]/IR[M]$ is called a generalized discrete Hodge algebra. For example, we write $M=(x^{2}, x y, y^{2})$ is a monoid generated by $x^{2}, xy$ and $y^{2}$. Let $I$ be the ideal generated by $xy$  inside $M$. Then $A=R[x^2,xy,y^2]/(xy)$ is a generalized discrete Hodge algebra.
%For an ideal $I$ in $M$, $I$ in $A[M]$ means the ideal generated by $I$ inside $A[M]$. So, we sometimes write $\frac{A[M]}{IA[M]}$ by $\frac{A[M]}{I}$. 
%It is clear that $A[M]/IA[M]\cong A[M/I]$. This identification is used in the main theorem.

  \item  \label{RingExtn:Item8}
  \textbf{(Split surjective morphisms)}
  A morphism $q_1: (A_1\subset B_1)\longrightarrow (A_2\subset B_2)$ of ring extensions is said to be a split surjection if there exists $q'_1: B_2\longrightarrow B_1$ such that $q_1\circ q'_1=\Id_{B_2}$ and the restriction $q'_{{1}_{|A_2}}: A_2\longrightarrow A_1$ such that $q_{{1}_{|A_1}}\circ q'_{{1}_{|A_2}} =\Id_{A_2}$.
  
  \begin{example}
  \label{examplesplitsurjection}
      Let $A\subset B$ be a ring extension. Define $A_1:=A[Y]$, $B_1:= B[Y]$, $A_2:= A$ and $B_2:= B$. The morphism $q_1: (A[Y]\subset B[Y])\longrightarrow (A\subset B)$ of ring extensions defined by 
      $q_1(Y)=0$ is a split surjective morphism in the category of ring extension because $q\circ i = \Id_{(A, B)}$, where the inclusion $i:  (A\subset B)\longrightarrow (A[Y]\subset B[Y])$ is defined by $b\mapsto b$. 
  \end{example}
 
 \end{enumerate}
\end{definition}

\begin{remark}
  Sometimes the operation is most conveniently written $+$ in which case the neutral element will be written as $0$. If need be the $+$ operation can be converted to multiplication by using exponentials. For example, in defining the monoid algebra $\mathbb{Z}\left[\mathbb{Z}_{+}\right]$the basis of $\mathbb{Z}\left[\mathbb{Z}_{+}\right]$ is the elements $\{0,1,2,3, \cdots\}$ of the monoid $\mathbb{Z}_{+}$ which we will represent respectively as $\left\{1, t, t^{2}, t^{3}, \cdots\right\}$ with addition $a+b$, for $a, b \in \mathbb{Z}_{+}$ corresponding to multiplication $t^{a} t^{b}=t^{a+b}, t$ a formal symbol or indeterminant. In this way the monoid algebra $\mathbb{Z}\left[\mathbb{Z}_{+}\right]$ becomes the polynomial ring $\mathbb{Z}[t]$. We leave it to the reader to make such translations. Thus we have $A\left[\mathbb{Z}^{n} \oplus \mathbb{Z}_{+}^{m}\right]=A\left[t_{1}, t_{1}^{-1} \ldots, t_{n}, t_{n}^{-1}\right] \otimes_{A} A\left[t_{1}, \ldots, t_{m}\right]$ for any commutative ring $A$. Another example comes from our citations of Theorems in \cite{B-G}. Theorems early in this book use additive notation, and it is convenient to revert to additive notation when applying them.
\end{remark}

The following result, due to Roberts and Singh, connects $\CI(A,B)$ with the Picard
and unit groups of $A$ and $B$.

\begin{proposition}\cite[Theorem $2.4$]{Roberts and Singh}
Let $A\subset B$ be an extension of rings. Then there exists a functorial exact sequence
{\small
\[\xymatrix{ 1 \ar@{->}[r] &\U(A) \ar@{->}[r]^{i} &\U(B) \ar@{->}[r]^{\theta} &\CI(A,B) \ar@{->}[r]^{\cl}
&\Pic(A)\ar@{->}[r]^{j} &\Pic(B)
}
\]
}
of abelian groups, where $i$ is an inclusion, $\theta(b) = Ab$, $\cl(I)$ is the class $[I]$ in $\Pic(A)$ and $j([I] = [I\otimes_{A}B]$.
\label{Lmexactseq}
\end{proposition}

\begin{lemma} \cite[Theorem $4.79$]{B-G} \label{Lm4}
Let $A\subset B$ be an extension of reduced rings and $M\subset N$ an extension of monoids. Then $A^{s}[M^s]$ is the subintegral closure of $A[M]$ in $B[N]$, 
where $M^s$ and $A^s$ denote the subintegral closure of 
$M$ in $N$ and of $A$ in $B$ respectively.
\end{lemma}
\begin{corollary}\label{Lm4Coro}
  Let $A\subset B$ be an extension of reduced rings and $M\subset N$ an extension of monoids.  If $A$ is subintegrally closed in $B$ and $M$ is subintegrally closed in $N$, then $A[M]$ is subintegrally closed in $B[N].$
\end{corollary}
\begin{proof}
   Follows from Lemma \ref{Lm4} because in this case $A^{s}=A$ and $M^{s}=M$.
    $\hfill \gj$
\end{proof}

\section{A theorem on subintegrality}
\label{On monoid subintegral extension}
We recall the definition of subintegral extension for monoids. Let $M$ and $N$ be two monoids. The monoid extension $M\subset N$ is  {\bf elementary subintegral} if $N=M\cup xM$ with $x^2,x^3\in M$.
  If $N=\cup N_{\gl}$, where $N_{\gl}$ can be obtained from $M$ by a finite number of elementary subintegral extensions, then the extension $M\subset N$ is {\bf subintegral}.

\begin{lemma}\label{ESub}
   If $N$ is elementary subintegral over $M$ and $z \in N$, then $z^{n} \in M$ for $n \geq 2$. 
\end{lemma}
\begin{proof}
   By definition of an elementary extension of monoids $N=y M$ for some $y \in N$ with $y^{2}, y^{3} \in M$. Then either $z \in M$ or $z=y w$ for some $w \in M$. If $z \in M$, then we are done. If $z=y w$, then $z^{n}=y^{n} w^{n}$ is always in $M$ because $y^{n}(n \geq 2)$ and $w^{n}$ are both in $M$. 
\end{proof}

\begin{proposition}\label{Prop8}
Let $M \subset N$ be an extension of monoids.
  The extension $M\subset N$ is subintegral $\Leftrightarrow$ $\forall z\in N$,  $z^j\in M \quad \forall \ j>>0$.
\end{proposition}
\begin{proof}
 $(\Rightarrow )$ Let $z\in N=\cup N_{\gl}$, hence $z\in N_{\gm}$ for some $\gm$. Hence we can write
 $M=N_1\subset N_2\subset \ldots \subset N_r=N_{\gm}$, where $N_i\subset N_{i+1}$, $1\leq i\leq r-1$ is an elementary subintegral extension, and $z\in N_{r}$. If $r=1$ then $z\in M$ we are done. Otherwise by a repetition of Lemma \ref{ESub} $z^{2^{r-1}} \in M$ and $z^{3^{r-1}} \in M$. Since $2^{r-1}$ and $3^{r-1}$ are relatively prime, by the Diophantine Frobenius problem (see page 31, Theorem 2.1.1, of \cite{Alfosin}) any integer greater than or equal to $i_{0}=2^{r-1} 3^{r-1}-2^{r-1}-3^{r-1}+1$ is of the form $A 2^{r-1}+B 3^{r-1}$ with $A, B$ integers $\geq 0$ which makes $z^{j} \in M$ if $j \geq i_{0}$, as desired.

 $(\Leftarrow)$ Let $z\in N$ such that $z^j\in M$ for all $j>>0$. Then there exists $i_0$ such that  $z^n\in M$ for all $n\geq i_0$. If $n_{0}=1$ then $z\in M$ and we are done. Otherwise take $v=z^{i_{0}-1}$. Then $v^2,v^3\in M$. Let $M_1=M\cup vM$. Then $M\subset M_1$ is an elementary subintegral extension. In this way, after finite steps, there exists a positive integer $r$ such that $M_{r}=M\cup zM$ and $M=M_0\subset M_1\subset M_2\subset \cdots \subset M_{r}$, where $M_i\subset M_{i+1}$ is an elementary subintegral extension for $i=0,1,\ldots, r-1$. Thus  there exists $M_{r}$ such that $M_{r}$ can be obtained from $M$ by a finite number of elementary subintegral extensions. Also we have $N=\cup M_{r}$. This completes the proof. 
 $\hfill \gj$
 %Let $M'$ be the monoid generated by $N-\{z\}$
 % Then $M'\subset N$ is an elementary subintegral extension.
 %Let $N_{z^j}$ be the monoid generated by $M$ and 
 %$z^j$ for $j=1,\ldots,i$. Define $N_{i-j}=M'\cap N_{z^j}$for $j=1,\ldots,i$, then $N_{0}=M$. Set $N_{i}=N_z$.
% It is easy to see that $N_k\subset N_{k+1}$ is an elementary subintegral extension, for $k=1,\ldots,i-1$.
% Clearly $N=\cup N_z$.
 
\end{proof}

%\begin{theorem}\cite[$(4.8),(4.17)$, and $(4.18)$]{Roberts and Singh}\label{Weak subintegrality}Let $A\subset B$ be an extension of $\BQ$-algebras. Then every element of $B$ is subintegral over $A$ if and only if $A\subset B$ is subintegral.
%\end{theorem}
\begin{theorem} \cite[Theorem $6.10$]{Reid Roberts Singhweak}
\label{Weakly Subintegrality Criterion}
 Let $A\subset B$ be an extension of commutative rings and let $b\in B$. Then the following are equivalent:\\
 (1)  $b$ is weakly subintegral over $A$.\\
 (2) $b$ satisfies any of the (equivalent) conditions:\\
   \begin{itemize}
       \item 
       There exists a positive integer $N$ and elements $c_1, \ldots , c_p\in B$ such that \[b^n+\sum\limits_{i=1}^{p} \binom{n}{i} c_ib^{n-i}\in A\,,for \ N\leq n\leq 2N+2p-1.\]
       \item
        There exists elements $c_1, \ldots ,   c_p\in B$ such that \[b^n+\sum\limits_{i=1}^{p} \binom{n}{i} c_ib^{n-i}\in A\,,for \ all \ n\geq 1.\]
        \item
        There exists  elements $c_1, \ldots , c_p\in B$ such that \[b^n+\sum\limits_{i=1}^{p} \binom{n}{i} c_ib^{n-i}\in A\,,for \ 1\leq n\leq 2p+1.\]
        \item
        There exists elements $c_1, \ldots , c_p\in B$ and  $N\in \mathbb{N}$, $s\in \mathbb{Z}^{+}$, $N\geq s+p$   such that \[b^n+\sum\limits_{i=1}^{p} \binom{n}{i} c_ib^{n-i-s}\in A\,, for \ all \ n\geq N.\]
   \end{itemize}
 (3)  $b$ satisfies the condition:\\
   There exists a positive integer $p\geq 0$ and elements $a_1, a_2, \ldots  a_{2p+1}\in A$ such that \[b^n+\sum\limits_{i=1}^{n} (-1)^i\binom{n}{i} a_ib^{n-i}=0 \, , for  \ p+1\leq n\leq 2p+1.\] 
  
\end{theorem}

%\begin{theorem}(%\cite[Theorem 3.22]{Vitulli2011},
%\cite[Theorem $2.1$]{Reid Roberts Singhweak}) Let $A\subset B$ be an extension of rings, $b\in B$ and $p$ a non negative integer. Then the following conditions are equivalent.\\
 % (i) The extension $A\subset A[b]$ is weakly subintegral.\\
 %(ii) There exists a positive integer $N$ and elements $c_1, \ldots , c_p\in B$ such that
% \[b^n+\sum\limits_{i=1}^{p} \binom{n}{i} c_ib^{n-i}\in A\,,for \ all \ n\geq N.\]
%(iii) There exists a positive integer $N$ and elements $c_1, \ldots , c_p\in B$ such that \[b^n+\sum\limits_{i=1}^{p} \binom{n}{i} c_ib^{n-i}\in A\,,for \ N\leq n\leq 2N+2p-1.\]
%(iv) There exists elements $c_1, \ldots ,   c_p\in B$ such that \[b^n+\sum\limits_{i=1}^{p} \binom{n}{i} c_ib^{n-i}\in A\,,for \ all \ n\geq 1.\]
%(v) There exists elements $a_1, a_2, \ldots \in A$ such that \[b^n+\sum\limits_{i=1}^{n} (-1)^i\binom{n}{i} a_ib^{n-i}=0 \,,for \ all \ n>p.\] 
 %(vi) There exists elements $a_1, a_2, \ldots  a_{2p+1}\in A$ such that \[b^n+\sum\limits_{i=1}^{n} (-1)^i\binom{n}{i} a_ib^{n-i}=0 \,,for  \ p+1\leq n\leq 2p+1.\] 
 
% \label{Weakly Subintegrality Criterion}
%\end{theorem}

\begin{proposition}
Let $A$ be a reduced ring and  $M\subset N$ an extension of monoids. Then $A[M]\subset A[N]$ is a subintegral extension
if and only if $M\subset N$ is a subintegral extension.
\end{proposition}

\begin{proof}
 $(\Leftarrow)$ This direction is clear.
 
 $(\Rightarrow)$ 
 %Let $z\in N$. Then $z\in R[N]$. By Theorem (\ref{Weak subintegrality}), $z^n+ \sum c_i z^{n-i}\in R[M]$. 
 %This shows that $z^n\in M,\, \forall \ n>>0$.
  Let $M'$ be the subintegral closure of $M$ in $N$. Then $M'$ is subintegrally closed in $N$. Hence, by Lemma \ref{Lm4}, 
 $A[M']$ is subintegrally closed in $A[N]$. Therefore $A[M']=A[N]$. Since $M'\subset N$ and $A[M']=A[N]$, we have $M'=N$. This proves that $M\subset N$ is a subintegral extension.
 $\hfill \gj$
\end{proof}
\vspace{.3cm}
\par
In the following, we generalize the above proposition by dropping the condition that the ring $A$ is reduced but we assume $\mathbb{Z}\subset A$. Interestingly, we prove that weak subintegrality and subintegrality are the same for the extension $A[M] \subset A[N]$. More precisely, we prove the following.

\begin{theorem} \label{twkly1}
Let $M\subset N$ be an extension of monoids.
Then the following are equivalent if $\BZ \subset A $, where $A$ is a ring.

$(1)$ The extension $A[M]\subset A[N]$ is subintegral.

$(2)$ The extension $A[M]\subset A[N]$ is  weakly subintegral.

$(3)$ The monoid extension $M\subset N$ is subintegral.
\end{theorem}

Before we delve into proving the theorem, we need the following technical lemma.

\begin{lemma}
If we choose $p+1$ distinct integers $0<a_1<a_2<\cdots <a_{p+1}$ then the
 matrix $M$ whose $i$-th row is  $\{ \binom{a_i}{0},\binom{a_i}{1},\ldots,\binom{a_i}{p}\}$,
 $1\leq i \leq p+1$ has determinant $\prod\limits_{1\leq i<j\leq p+1}(a_j-a_i)/p!(p-1)!\ldots 2!$, which is a positive integer.
 \label{Lemma3}
\end{lemma}
\begin{proof}
Let $M$ be the matrix whose $j$-th column is given by $\left\{ \binom{a_1}{j},\binom{a_2}{j},\ldots,\binom{a_{p+1}}{j}\right\}$ for $0\leq j\leq p$. Since each of the entries of $M$ is an integer, we have $\det(M)$ is an integer. Also note that column $0$ has  entries $1$ and column $1$ has  entries $a_{i}$, $1\leq i\leq p+1$. Let us apply elementary column operations to the matrix

   \begin{center}
    $ M = \begin{bmatrix}
           1 & a_1 & \frac{a_1(a_1-1)}{2!} &\ldots & \frac{a_1(a_1-1)\ldots (a_1-(j-1))}{j!} & \ldots & \frac{a_1(a_1-1) \ldots (a_1-(p-1))}{p!}\\
            1 & a_2 & \frac{a_2(a_2-1)}{2!} &\ldots & \frac{a_2(a_2-1)\ldots (a_2-(j-1))}{j!} & \ldots & \frac{a_2(a_2-1) \ldots (a_2-(p-1))}{p!}\\
           \vdots & \vdots & \vdots & \ddots  & \vdots & \ddots &    \vdots  \\
            1 & a_{p+1} & \frac{a_{p+1}(a_{p+1}-1)}{2!} &\ldots & \frac{a_{p+1}(a_{p+1}-1)\ldots (a_{p+1}-(j-1))}{j!} & \ldots & \frac{a_{p+1}(a_{p+1}-1) \ldots (a_{p+1}-(p-1))}{p!}\\
          \end{bmatrix}$ \vspace{.5cm}
          $2\leq j\leq p$ $\Big\downarrow$ $C_{j}' = C_{j}-\sum \limits_{i=0}^{j-1}(i-1)!C_{i-1} $
          \\ \vspace{.5cm}
           $\begin{bmatrix}
           1 & a_1 & \frac{a_1^2}{2!} &\ldots & \frac{a_1^j}{j!} & \ldots & \frac{a_1^p}{p!}\\
            1 & a_2 & \frac{a_2^2}{2!} &\ldots & \frac{a_2^j}{j!} & \ldots & \frac{a_2^p}{p!}\\
            \vdots & \vdots & \vdots & \ddots & \vdots &\ddots  & \vdots \\
            1 & a_{p+1} & \frac{a_{p+1}^2}{2!} &\ldots & \frac{a_{p+1}^j}{j!} & \ldots & \frac{a_{p+1}^p}{p!}\\
          \end{bmatrix} := W$
   \end{center}
 Then we get $\det(M) = \det(W) = \frac{\det(V)}{p!(p-1)!\ldots 2!}$, where $V$ is a Vandermonde matrix whose $j$-th column is  $\{a_1^j, a_2^j, \ldots , a_{p+1}^j\}$, $0\leq j\leq p$.  It is known that $V$ has a determinant $\prod \limits_{1\leq i< j\leq p+1 }(a_j-a_i)$. Hence we have $\det(M) =  \prod \limits_{1\leq i<j\leq p+1}\frac{(a_j-a_i)}{p!(p-1)!\ldots 2!}$, which is a positive integer.
$\hfill \gj$
\end{proof}

\vspace{.3cm}
{\bf Proof of Theorem \ref{twkly1}:}
\par
$(1)\Rightarrow (2)$ This is always true.
\par
$(2)\Rightarrow (3)$   Let $z\in N$. Then by Theorem \ref{Weakly Subintegrality Criterion} there exists a positive integer $p$ and elements $c_{1}$, $c_{2}$, \ldots , $c_{p}\in A[N]$ such that $z^{n} + \sum\limits_{i=1}^{p}\binom{n}{i}c_{i}z^{n-i}\in A[M]$, $n\gg 0$. 
$A[N]$ is a free $A$-module with basis $N$ and if we look at the $z^{n}$ component we can assume that  $c_{i} = c_{i}'z^{i}$ for $c_{i}'\in A$, $1\leq i\leq p$, yielding $z^{n} \left( 1 + \sum\limits_{i=1}^{p}\binom{n}{i}c_{i}'\right)\in A[M]$ for all $n\gg 0$. The elements 
$ 1 + \sum\limits_{i=1}^{p}\binom{n}{i} c_{i}'\in A$ can be $0$ for at most $p$ values of $n$ because if we choose any distinct $p+1$ values of $n$ the determinant of the coefficient matrix  $\left(\binom{n}{i}\right)$ is a non zero integer by Lemma \ref{Lemma3} . Therefore  $z^{n}\in M$ for all $n\gg 0$.

$(3)\Rightarrow (1)$  This follows from \ref{Lm4Coro}.
$\hfill \gj$

\begin{example}\label{counter-example}
Here we provide an example that shows that Theorem \ref{twkly1} is not true if $A$ is an $\mathbb{F}_{p}$-algebra, where $p$ is a prime number and $\mathbb{F}_p$ is a field with $p$-elements.
Consider the monoid extension $\mathbb{F}_{p}[M]\subset \mathbb{F}_{p}[N]$,
where $M=(x^p)$ and $N=(x)$. Then it is easy to see that
the above extension is weakly subintegral. However, it is not a subintegral extension. 
\end{example}

\begin{example}\label{counter-z}
We provide an example that shows that Theorem \ref{twkly1} is not true if $A$ is a $\mathbb{Z}$-algebra but not a monoid algebra.
Consider the $\mathbb{Z}$-algebra extension $\mathbb{Z}[2t,t^2] \subset \mathbb{Z}[t]$. Since $t^2, 2t \in \mathbb{Z}[2t, t^2]$,
the extension $\mathbb{Z}[2t,t^2]\subset \mathbb{Z}[t]$ is weakly 
subintegral. It follows from \cite[Theorem 3.3]{RR00} that the extension $\mathbb{Z}[2t,t^2]\subset \mathbb{Z}[t]$ is not subintegral.
\end{example}

\section{Invertible modules over generalized discrete Hodge algebras}\label{Sec_4}

Some more notation: \textbf{Rings}, \textbf{A-Mod} ($A$ is a ring), \textbf{AbGr}  denote the category of commutative rings, category of $A$-modules, and category of abelian groups respectively.

\subsection{Milnor square in the category of ring extension}

\begin{definition} \cite[Section $7.2$, Definition $2$] {BoschBook}
Let $\CC$ be a category. Let $X, Y, Z$ be  three objects in the category $\CC$ and let $q_1: X\longrightarrow Z$,  $q_2: Y\longrightarrow Z$ be morphisms in $\CC$. The Cartesian square (also called pullback or fiber product diagram) of objects $X$ and $Y$ over the object $Z$ is a triple $(P, p_1, p_2)$, where $P$ is an object in $\CC$, $p_1: P\longrightarrow X$ and $p_2: P\longrightarrow Y$  are morphisms such that $q_1\circ p_1 = q_2\circ p_2$ and the triple $(P, p_1, p_2)$
is universal in the sense that given any other triple $(P', p'_1, p'_2)$ of this kind
with $q_1\circ p'_1 = q_2\circ p'_2$ there is a unique morphism $h: P'\longrightarrow P$ such that $p_1\circ h = p'_1$ and $p_2\circ h = p'_2$. Pictorially, the following diagram
\begin{center}
\begin{tikzcd}
P' \arrow[rrd, "p'_1", bend left] \arrow[rdd, "p'_2"', bend right] \arrow[rd, "\exists ! h", dashed] &                                     &                    \\
   & P \arrow[d, "p_2"] \arrow[r, "p_1"] & X \arrow[d, "q_1"] \\
   & Y \arrow[r, "q_2"]                  & Z                 
\end{tikzcd}
\end{center}
is commutative in the category $\CC$.
\end{definition}

\begin{example}\label{LmFibProd_1}
Let $A$ be a commutative ring and let $I, J$ be two ideals of $A$. Then
\begin{center}
\begin{tikzcd}
A/({I\cap J}) \arrow[d,"p_2"] \arrow[r, "p_1"] & A/I \arrow[d, "q_1"] \\
A/J \arrow[r, "q_2"]                   & A/({I+J})      
\end{tikzcd}
\end{center}
with maps $p_1$, $p_2$, $q_1$, $q_2$ are natural, is a Cartesian square in the category $\textbf{Rings}$.\\
Further, if $M$ is an $A$-module then
\begin{center}
\begin{tikzcd}
M/({IM\cap JM}) \arrow[d,"p_2"] \arrow[r, "p_1"] & M/IM \arrow[d, "q_1"] \\
M/JM \arrow[r, "q_2"]                   & M/{(I+J)M}      
\end{tikzcd}
\end{center}
with maps $p_1$, $p_2$, $q_1$, $q_2$ are natural, is a Cartesian square in the category $\textbf{A-Mod}$.
Since the above two examples are well-known, we leave the details to the reader.
\end{example}
%\begin{proof}
%Proof is standard.
%\end{proof}

 \begin{lemma}\label{LmFibProdEqv}
 Let $\CC$ be $\textbf{Rings}$ or $\textbf{A-Mod}$ or $\textbf{AbGr}$  and let $X, Y, Z$ be three objects in $\CC$. Consider the commutative diagram
 \begin{center}
\begin{tikzcd}
P \arrow[d,"p_2"] \arrow[r, "p_1"] & X \arrow[d,"q_1"] \\
Y \arrow[r, "q_2"]                   & Z     
\end{tikzcd}
\end{center}
in  the category $\CC$. This diagram is a Cartesian square if and only if for each pair of elements $x$ in $X$ and $y$ in $Y$ with $q_1(x) = q_2(y)$, there is a unique element $u$ in $P$ with $p_1(u) = x$ and
$p_2(u) = y$.
 \end{lemma}
\begin{proof}
This is easily proved and well-known.
 $\hfill \gj$
\end{proof}

\begin{definition}({\bf Milnor square in the category of rings.})

The following Cartesian square of rings

\[
 \xymatrix{
  A \ar@{->}[r]^{p_1}\ar@{->}[d]^{p_2} & A_1  \ar@{->}[d]^{q_1}\\
   A_2 \ar@{->}[r]^{q_2} & A_3
   }
\]
is called a Milnor square if $q_1$ is a surjective ring homomorphism. For details, see page $14$ of \cite{CW2011}.

\end{definition}

Recall that morphisms in the category of ring extensions are defined in Definition \ref{Weakly subintegral extensions}(\ref{RingExtn:Item5}).
\begin{definition}({\bf Milnor square in the category of ring extensions})\label{Mil_Sqr_Ring_Ext}
 Let $A\subset B$, $A_1\subset B_1$, $A_2\subset B_2$ and $A_3\subset B_3$ be ring extensions
 such that the following squares  
\begin{equation}\label{eqn:AB}
 %\[
\xymatrix{
 B\ar@{->}[r]^{p_1}          
     \ar@{->}[d]^{p_2}
&B_1 
     \ar@{->}[d]^{q_1}       && A\ar@{->}[r]^{p_{{1}_{|A}}}\ar@{->}[d]^{p_{{2}_{|A}}} & A_1 \ar@{->}[d]^{q_{{1}_{|A_1}}}
\\
B_2 \ar@{->}[r]^{q_2}
     &B_3                && A_2\ar@{->}[r]^{q_{{2}_{|A_2}}} & A_3 
}
%\]
\end{equation}
are Milnor squares, where the notation $f_{|A}$ means the restriction of $f$ to $A$ and the maps $p_1$, $p_2$, $q_1$, $q_2$ are morphisms in the category of ring extensions.

Then we call the following square 
\[
 \xymatrix{
  A\subset B\ar@{->}[r]^{p_1}\ar@{->}[d]^{p_2} & A_1\subset B_1 \ar@{->}[d]^{q_1}\\
   A_2\subset B_2\ar@{->}[r]^{q_2} & A_3\subset B_3
   }
\]
 a Milnor square in the category of ring extensions. 
 
\end{definition}

\begin{proposition} \label{Bass-Sequence}
Let  
\[
\xymatrix{
A\ar@{->}[r]^{p_1}         
     \ar@{->}[d]^{p_2}
&A_1 
     \ar@{->}[d] ^{q_1}     
\\
A_2 \ar@{->}[r]^{q_2}
     &A_3               
}
\]
be a Milnor square in the category of rings. We have the following
{\small
\[\xymatrix{ 1 \ar@{->}[r] &\U(A) \ar@{->}[r]^(.4){p} &\U(A_1)\oplus \U(A_2) \ar@{->}[r]^(.6){q} &\U(A_3) \ar@{->}[r]^{h}
&\Pic(A)\ar@{->}[r]^(.3){\alpha} &\Pic(A_1)\oplus \Pic(A_2) \ar@{->}[r]^(.6){\beta} &\Pic(A_3)
}
\]
}
 $6$-terms exact sequence of abelian groups, where $p(a)=(p_1(a), p_2(a))$, $q(a_1,a_2)=q_1(a_1)q_2(a_2)^{-1}$, $\alpha([P])=([P\tens{A}A_1], [P\tens{A}A_2])$, $\beta([P],[Q])= [P]\tens{A_3}[Q^{-1}]$ and the map $h$ is the connecting homomorphism.
\end{proposition}
\begin{proof}
 This follows from \cite[Chapter IX, Theorem 5.3]{BassBook}.
\end{proof}
%%%%%%%%%%%%%%%%%%%%%%%%%%%%%%%%%%%%%%%%%%%%%%%%%%%%%%%%%%
% \begin{lemma}\label{prop:units}
% For any two ring extensions $A_1\subset B_1$ and $A_2\subset B_2$, we have the following  
% \[ \xymatrix{1\ar@{->}[r] & \U(A_1)\oplus \U(A_2) \ar@{->}[r] &\U(B_1)\oplus \U(B_2) \ar@{->}[r]^(.4){\theta} &\CI(A_1, B_1)\oplus \CI(A_2, B_2)
% }
% \]
%  exact sequence of abelian groups, where the map $\theta$ is defined by $\theta(b_1, b_2) = (b_1A_1, b_2A_2)$. Further, for any $(a_1,a_2)\in \U(A_1)\oplus \U(A_2)$ and $(b_1, b_2)\in \U(B_1)\oplus \U(B_2)$, we have  $\theta(a_1b_1, a_2b_2) = \theta(b_1,b_2)$. 
% \end{lemma}
% \begin{proof}
%     Applying Proposition \ref{Lmexactseq}
%  to the ring extensions $A_1\subset B_1$ and $A_2\subset B_2$, we have two exact sequences 
% \[\xymatrix{1\ar@{->}[r] &\U(A_1)\ar@{->}[r] &\U(B_1)\ar@{->}[r] &\CI(A_1, B_1)
% }
% \]
% and
% \[\xymatrix{1\ar@{->}[r] &\U(A_2) \ar@{->}[r] & \U(B_2) \ar@{->}[r] &\CI(A_2, B_2).
% }
% \]
% Now  taking direct sum, we get the required exact sequence
% \[\xymatrix{1\ar@{->}[r] &\U(A_1)\oplus \U(A_2) \ar@{->}[r] &\U(B_1)\oplus \U(B_2) \ar@{->}[r]^(.4){\theta} &\CI(A_1, B_1)\oplus \CI(A_2, B_2).
% }
% \]
% For any $(a_1, a_2)\in \U(A_1)\oplus \U(A_2)$,  we get $\theta(a_1b_1, a_2b_2) = (a_1b_1A_1, a_2b_2A_2) = (b_1A_1, b_2A_2) = \theta(b_1, b_2)$.
% $\hfill \gj$
% \end{proof}

\begin{notation}\label{notn:square}
 In the proof of the next lemma, by a $\psi$-square, we mean the following commutative diagram 
\[
 \xymatrix{
  G_1 \ar@{->}[r]^{\psi}\ar@{->}[d]^{\phi} & H_1  \ar@{->}[d]^{\alpha }\\
   G_2 \ar@{->}[r]^{\beta} & H_2
   }
\]
of abelian groups,  i.e., the diagram is named by its top arrow.
   
\end{notation}

\vspace{.3cm}
Recall that Milnor squares in the category of ring extensions are defined in Definition \ref{Mil_Sqr_Ring_Ext}, split surjective morphisms are defined in Definition \ref{Weakly subintegral extensions}(\ref{RingExtn:Item8}) and Example \ref{examplesplitsurjection} is an example of a split surjective morphism.

\begin{lemma} \label{7l1}
Let the following commutative diagram 
\begin{center}
\begin{tikzcd}
A\subset B \arrow[d, "p_2"] \arrow[r, "p_1"] & A_1\subset B_1 \arrow[d, "q_1"] \\
A_2\subset B_2 \arrow[r, "q_2"]              & A_3\subset B_3                 
\end{tikzcd}
\end{center}
be a Milnor square in the category of ring extensions, where $p_1$, $p_2$, $q_1$, $q_2$ are morphisms in the category of ring extensions and $q_1$ is a split surjective morphism. Then we have the following exact sequence of abelian groups
\begin{equation}
\label{3term exact seq}
      \begin{tikzcd}
     1\arrow{r}& \mathcal{I}(A,B)\arrow{r}{\phi}& \mathcal{I}(A_{1},B_{1})\oplus \mathcal{I}(A_{2}, B_{2})\arrow{r}{\psi} & \mathcal{I}(A_{3}, B_{3}) \arrow{r}{} & 1
     \end{tikzcd},
     \end{equation}
    where $\phi(M)= (\CI(p_1)(M), \CI(p_2)(M)) = 
    (p_1(M)A_1,p_2(M)A_2)$, 
    $\psi(M_1,M_2)= \CI(q_1)(M_1)(\CI(q_2)(M_2))^{-1}$, $M \in \CI(A,B)$, $M_1 \in \CI(A_1,B_1)$, and $M_2 \in \CI(A_2,B_2)$.
\end{lemma}

\begin{proof}
  Since $q_1p_1=q_2p_2$, we observe that $\IIm(\phi)\subset \ker(\psi)$. For the reverse inclusion, first, we observe that 
  the $6$-term exact sequence in Proposition \ref{Bass-Sequence}  breaks into the following  
  two exact sequences
\[\xymatrix{ 1 \ar@{->}[r] &\U(A) \ar@{->}[r] &\U(A_1)\oplus \U(A_2) \ar@{->}[r] &\U(A_3) \ar@{->}[r] &1
}
\]

\[\xymatrix{ 1 \ar@{->}[r]
&\Pic(A)\ar@{->}[r] &\Pic(A_1)\oplus \Pic(A_2) \ar@{->}[r] &\Pic(A_3)\ar@{->}[r]&1
}
\]
 as $q_1$ is a split surjection (by abuse of notation we are writing $q_{1}: A_{1} \rightarrow A_{3}$ instead of $\left.q_{1}\right|_{A_{1}}: A_{1} \rightarrow A_{3}$).
 There are two similar exact sequences with $A$ 's replaced by $B$'s. Now we have the following diagram 
  
  \begin{center}
\begin{tikzcd}
 & 1 \arrow[d]                       & 1 \arrow[d]                                       & 1 \arrow[d]        \\

1 \arrow[r] & \U(A) \arrow[r, "\delta"] \arrow[d, "i_1"]                       & \U(A_1)\oplus \U(A_2) \arrow[r, "\gamma"] \arrow[d, "i_2"]                                       & \U(A_3) \arrow[r,] \arrow[d, "i_3"]       & 1 \\
1 \arrow[r] & \U(B) \arrow[r, "p"] \arrow[d, "\theta"]              & \U(B_1)\oplus \U(B_2) \arrow[r, "q"] \arrow[d, "{\widetilde{\theta} = (\theta_1, \theta_2)}"]       & \U(B_3) \arrow[r] \arrow[d, "\theta_3"] & 1 \\
            1\arrow{r} & {\mathcal{I}(A,B)} \arrow[d, "\cl"] \arrow[r, "\phi"] & {\mathcal{I}(A_1,B_1)\oplus \mathcal{I}(A_2,B_2)} \arrow[d, "{\widetilde{\cl} = (\cl_1, \cl_2)}"] \arrow[r] \arrow[r, "\psi"] & {\mathcal{I}(A_3,B_3)} \arrow[r, " "]\arrow[d, "\cl_3"] & 1  \\
1 \arrow[r] & \Pic(A) \arrow[r, "\alpha"] \arrow[d, "j_1"]            & \Pic(A_1)\oplus \Pic(A_2) \arrow[r, "\beta"] \arrow[d, "j_2"]                                     & \Pic(A_3) \arrow[d, "j_3"] \arrow{r} &1               &   \\
1 \arrow[r] & \Pic(B) \arrow[r, "\alpha'"]                          & \Pic(B_1)\oplus \Pic(B_2) \arrow[r, "\beta'"]                                                   & \Pic(B_3) \arrow{r}& 1                               &  
\end{tikzcd}

  \end{center}
where maps are defined by $p(b)=(p_1(b), p_2(b))$, $p|_{\U(A)}$ is the restriction of $p$ to $\U(A)$, $q(b_1, b_2) = q_1(b_1)(q_2(b_2))^{-1}$, $\gamma(a_1, a_2) = q_1(a_1)(q_2(a_2))^{-1}$,
$\cl(P)= [P]$, $j_1([P])= [P\tens{A}B]$. Similarly we can define the maps $\widetilde{\cl}$, $\cl_3$, $j_2$, $j_3$ and $\alpha$, $\beta$, $\alpha'$, $\beta'$ are defined as in Proposition \ref{Bass-Sequence}. All squares are commutative because all maps are functorial and all columns are exact sequences because of Propositions \ref{Lmexactseq}. The fact that $q_1$ is a  split surjective morphism makes it clear that all rows are exact except the third one.

Let $x\in \ker(\psi)$, which implies that $\cl_3(\psi(x))=1\in \Pic(A_3)$. By commutativity of the $\psi$-square (refer to Notation \ref{notn:square}) we get $\beta(\widetilde{\cl}(x))=1\in \Pic(A_3)$. By  the exactness of the $4$-th row there exists $z\in \Pic(A)$ such that  $\alpha(z)=\widetilde{\cl}(x)$. By commutativity of the $\alpha$-square we have $\alpha^{\prime} j_{1}=j_{2} \alpha$. Therefore $\alpha^{\prime} j_{1}(z)=j_{2} \alpha(z)=j_{2} \tilde{c}(x)=1$. By injectivity of $\alpha^{\prime}$ we have $j_{1}(z)=$ $1 \in \operatorname{Pic}(B)$. By exactness of the first column there exists $t \in \mathcal{I}(A, B)$ such that $\cl(t)=z$. By commutativity of the $\phi$-square $\tilde{\cl} \phi=\alpha \cl$ so $\tilde{\cl} \phi(t)=\alpha \cl(t)=\alpha(z)=\tilde{\cl}(x)$. Therefore $\phi(t)$ and $x$ have the same image under $\tilde{\cl}$ so by exactness of the second column there exists $y^{\prime} \in \mathrm{U}\left(B_{1}\right) \oplus \mathrm{U}\left(B_{2}\right)$ such that $\tilde{\theta}\left(y^{\prime}\right)=x \phi(t)^{-1}$.
Let $z^{\prime}=q\left(y^{\prime}\right) \in \mathrm{U}\left(B_{3}\right)$. By commutativity of the $q$-square $\theta_{3} q=\psi \tilde{\theta}$ so $\theta_{3}\left(z^{\prime}\right)=\theta_{3} q\left(y^{\prime}\right)=\psi \tilde{\theta}\left(y^{\prime}\right)=\psi\left(x \phi(t)^{-1}\right)=\psi(x)(\psi \phi(t))^{-1}=$ $1 \cdot 1=1 \in \mathcal{I}\left(A_{3}, B_{3}\right)$. By exactness of the third column, there exists $t^{\prime} \in \mathrm{U}\left(A_{3}\right)$ such that $i_{3}\left(t^{\prime}\right)=z^{\prime}$. Since $\gamma$ is onto, there exists $u \in$ $\mathrm{U}\left(A_{1}\right) \oplus \mathrm{U}\left(A_{2}\right)$ such that $\gamma(u)=t^{\prime}.$ By the commutativity of the $\gamma$-square we have $z^{\prime}=i_{3} \gamma(u)=q\left(i_{2}(u)\right)$. But by definition $z^{\prime}=q\left(y^{\prime}\right)$. So $y^{\prime}$ and $i_{2}(u)$ have the same image under $q$, which can be written $q\left(y^{\prime} i_{2}(u)^{-1}\right)=1.$ 
By exactness of the second row there exists $v \in \mathrm{U}(B)$ such that $p(v)=y^{\prime}\left(i_{2}(u)\right)^{-1}$. Now set $w=\theta(v)$. By commutativity of the $p$ square $\tilde{\theta} p=\phi \theta$. Therefore $\phi(w)=\phi \theta(v)=\tilde{\theta} p(v)=\tilde{\theta}\left(y^{\prime}\left(i_{2}(u)\right)^{-1}\right)=\tilde{\theta}\left(y^{\prime}\right)\left(\tilde{\theta} i_{2}(u)\right)^{-1}=\tilde{\theta}\left(y^{\prime}\right)$. By definition of $y^{\prime}, \tilde{\theta}\left(y^{\prime}\right)=x \phi(t)^{-1}$. Thus $x=\phi(t) \phi(w)=\phi(t w)$ as we wished to show.

Now we prove $\phi$ is injective. For that let $x \in \mathcal{I}(A, B)$ with $\phi(x)=1$. Let $y=\cl(x)$. By commutativity of the $\phi$-square $\alpha(y)=1$,
and since $\alpha$ is injective $y=1$. By exactness of the first column there exists $z \in \mathrm{U}(B)$ such that $\theta(z)=x$. Let $t=p(z) \in \mathrm{U}\left(B_{1}\right) \oplus \mathrm{U}\left(B_{2}\right)$. By exactness of the second row, $q(t)=q p(z)=1$. By commutativity of the $p$-square $\tilde{\theta} p=\phi \theta$ so $\tilde{\theta}(t)=\tilde{\theta} p(z)=\phi \theta(z)=\phi(x)=1$. By exactness of the second column there exists $u \in \mathrm{U}\left(A_{1}\right) \oplus \mathrm{U}\left(A_{2}\right)$ such that $i_{2}(u)=t$. By commutativity of the $\gamma$-square we have $i_{3} \gamma=q i_{2}$ so $i_{3} \gamma(u)=q i_{2}(u)=q(t)=1$. Since $i_{3}$ is an inclusion $\gamma(u)=1$. By exactness of the first row there exists $w \in \mathrm{U}(A)$ so that $\delta(w)=u$. By commutativity of the $\delta$-square we have $i_{2} \delta=p i_{1}$ so $t=i_{2}(u)=$ $i_{2} \delta(w)=p i_{1}(w)$. But also $p(z)=t$ and $p$ is an inclusion so $z=i_{1}(w)$. By exactness of the first column $x=\theta(z)=\theta i_{1}(w)=1$ as we wished to show.

Now, it remains to show that the map $\psi$ is surjective. Since the map $q_1$ is a split surjection, there exists a map $q_1': (A_3, B_3)\rightarrow (A_1, B_1)$ such that $q_1\circ q_1'=\Id_{(A_3, B_3)}$. By applying the functor $\CI$ we get $\CI(q_1)\circ \CI(q_1')=\Id_{\CI(A_3, B_3)}$, which gives the map $\CI(q_1): \CI(A_1, B_1)\rightarrow \CI(A_3,B_3)$ is surjective. Thus for any $J\in \CI(A_3,B_3)$, there there exists $M_1\in \CI(A_1,B_1)$ such that $\CI(q_1)(M_1)=J$. Then $\psi(M_1, A_2)=\CI(q_1)(M_1)(\CI(q_1)(A_2))^{-1}=J$, hence $\psi$ is surjective. This completes the proof.
$\hfill \gj$
\end{proof}

\begin{remark}
In the beginning, our claim was to prove that given a Milnor square in the category of ring extensions, line (\ref{3term exact seq}) in the statement of Lemma \ref{7l1} is an exact sequence but we prove this with the assumption that $q_1$ is a split surjection. 
The split surjection of $q_1$ needed to conclude that $\U(A_1)\longrightarrow \U(A_3)$ and $\U(B_1)\longrightarrow \U(B_3)$ are surjective and to prove $\psi$ is surjective.
The split surjection of $q_{1}$ also implies that the maps $\Pic\left(A_{1}\right) \rightarrow$ $\Pic\left(A_{3}\right)$ and $\Pic\left(B_{1}\right) \rightarrow \Pic\left(B_{3}\right)$ are surjective, hence $\beta$ and $\beta^{\prime}$ are surjective. However, we did not used these facts in the above proof.
\end{remark}

\subsection{Application of Milnor square for extension of monoid algebras}

\begin{lemma}\label{Lm2}
Let $A\subset B$ be an extension of reduced rings and $M\subset N$ an extension of affine monoids such that $A$ is subintegrally closed in $B$ and $M$ is subintegrally closed in $N$. Let $I$ be a radical ideal of $N$.
Then $\frac{A[M]}{(I\cap M)A[M]}$ is subintegrally closed in $\frac{B[N]}{IB[N]}$.
\end{lemma}
\begin{proof}
%Since $A$ is subintegrally closed in $B$ and $M$ is subintegrally closed in $N$, by Corollary \ref{Lm4Coro} we have $A[M]$ is subintegrally closed in $B[N]$. 
Because $IB[N] \cap A[M]=(I \cap M) A[M]$, we have an inclusion $i_{I}: \frac{A[M]}{(I \cap M) A[M]} \hookrightarrow$ $\frac{B[N]}{I B[N]}$. Since $I$ is a radical ideal, by \cite[Proposition 2.36]{B-G}, we can assume that $I=\p_1\cap \ldots \cap \p_n$, where $\mathfrak{p_i}$'s are prime ideals in $N$. We prove the result by induction on $n$. 
\par
If $n = 1$, then $I = \mathfrak{p}_1$. Define $N' := N \setminus \mathfrak{p}_1$ and $M' := M \setminus (\mathfrak{p}_1 \cap M)$. Then $N'$ is a submonoid of $N$, and $M'$ is a submonoid of $M$. Moreover, since $M$ is subintegrally closed in $N$ by assumption, it follows that $M'$ is subintegrally closed in $N'$.
Observe that $\frac{A[M]}{(I \cap M) A[M]} \cong A\left[M^{\prime}\right]$ because both sides are the free $A$-module with basis those elements $m \in M$ that are not contained in $I \cap M$. Similarly we have $B\left[N^{\prime}\right] \cong \frac{B[N]}{I B[N]}$. By Corollary \ref{Lm4Coro} $A\left[M^{\prime}\right]$ is subintegrally closed in $B\left[N^{\prime}\right]$ so $A\left[M^{\prime}\right] \cong$ $\frac{A[M]}{(I \cap M) A[M]}$ is subintegrally closed in $B\left[N^{\prime}\right] \cong \frac{B[N]}{I B[N]}$. This establishes the case $n=1$.

 Assume that $n>1$. Let $J=\mathfrak{p}_{2} \cap \ldots \cap \mathfrak{p}_{n}$. Note that $J B[N]+$ $\mathfrak{p}_{1} B[N]=\left(J \cup \mathfrak{p}_{1}\right) B[N]$. Similarly $(J \cap M) A[M]+\left(\mathfrak{p}_{1} \cap M\right) A[M]=$ $\left(\left(J \cup \mathfrak{p}_{1}\right) \cap M\right) A[M]$. This leads to the following Milnor square in the category of ring extensions (as defined in Definition \ref{Mil_Sqr_Ring_Ext}):
\[
\xymatrix{
  \frac{A[M]}{(I\cap M)A[M]}\subset  \frac{B[N]}{IB[N]} \ar[r]^{p_1} \ar@<-2pt>[d]_{p_2} & \frac{A[M]}{(J\cap M)A[M]}\subset \frac{B[N]}{JB[N]} \ar@<-2pt>[d]_{q_1} \\
  \frac{A[M]}{(\p_1\cap M)A[M]} \subset  \frac{B[ N]}{\p_1B[N]} \ar[r]^(.38){q_2} & \frac{A[M]}{((J\cup \p_1)\cap M)A[M]} \subset  \frac{B[N]}{(J\cup \p_1)B[N]}}.
\]

This consists of two Milnor squares in the category of commutative rings, one to the left of the inclusions, and one to the right, with the inclusions $\subset$ giving a map between them. The left Milnor square is constructed as in Example \ref{LmFibProd_1} with $A$ in Example \ref{LmFibProd_1} replaced by $A[M], I$ replaced by $(J \cap M) A[M]$ and $J$ replaced by $\left(\mathfrak{p}_{1} \cap M\right) A[M]$. The right Milnor square is similarly constructed from Example \ref{LmFibProd_1} with $A$ replaced by $B[N], I$ replaced by $J B[N]$ and $J$ replaced by $\mathfrak{p}_{1} B[N]$. The upper left $\subset$ was given a name $i_{I}$, which is an inclusion $i_{I}: \frac{A[M]}{(I \cap M) A[M]} \hookrightarrow \frac{B[N]}{I B[N]}$. Similarly inclusions $i_{J}, i_{\mathfrak{p}_{1}}$, and $i_{J \cup P_{1}}$ are defined, each indicated by $\subset$ at the appropriate place in the diagram. At the coset level $i_{I}(a+(I \cap M) A[M])=a+I B[N]$ for some $a \in A[M]$. The existence or not of such an $a$ may not be apparent if we are given a coset $b+I B[N] \in \frac{B[N]}{I B[N]}$, and if $a$ exists it will not be unique. And as sets $(a+(I \cap M) A[M])$ and its image $a+I B[N]$ are different, the first being a copy of $(I \cap M) A[M]$ and the second of $I B[N]$. To simplify notation we may abusively write $f \in \frac{B[N]}{I B[N]}$ instead of $f+I B[N], f \in B[N]$. If $f \in \operatorname{Im}\left(i_{I}\right)$ it is convenient to write $f \in \frac{A[M]}{(I \cap M) A[M]}$ which will not mean that $f \in A[M]$, but rather that the coset $f+I B[N]$ contains an element of $A[M]$. The restriction of $p_{1}$ to $\frac{A[M]}{I \cap M) A[M]}$ will be denoted by $p_{1 \mid}$. Similarly we write $p_{2 \mid}, q_{1 \mid}$, and $q_{2\mid}$.

We wish to prove that $\frac{A[M]}{(I \cap M) A[M]}$ is subintegrally closed in $\frac{B[N]}{I B[N]}$. This will be the case if every element $f \in \frac{B[N]}{I B N}$ such that $f^{k} \in \frac{A[M]}{(I \cap M) A[M]}, k \geq 2$ is already in $\frac{A[M]}{(I \cap M) A[M]}$. So start with such an $f \in B[N]$ (variously thought of as $f \in \frac{B[N]}{I B[N]}$ or as the coset $f+I B[N]$). Then we have $p_{1}(f)=f+J B[N]$ and for $k \geq 2$ we have $p_{1}\left(f^{k}\right)=$ $p_{1}(f)^{k} \in \frac{A[M]}{(J \cap M) A[M]}$. By the induction hypothesis $\frac{A[M]}{(J \cap M) A[M]}$ is subintegrally closed in $\frac{B[N]}{J B[N]}$ so $p_{1}(f) \in \frac{A[M]}{(J \cap M) A[M]}$. Similarly $\frac{A[M]}{\left(\left(\mathfrak{p}_{1}\right) \cap M\right) A[M]}$ is subintegrally closed in $\frac{B[N]}{\mathfrak{p}_{1} B[N]}$ so $p_{2}(f) \in \frac{A[M]}{\left(\mathfrak{p}_{1} \cap M\right) A[M]}$. Therefore $p_{1}(f)=$ $f+J B[N]=a_{1}+J B[N]$ and $p_{2}(f)=f+\mathfrak{p}_{1} B[N]=a_{2}+\mathfrak{p}_{1} B[N]$, for some $a_{1}, a_{2} \in A[M]$. These must have the same image when mapped into the lower right corner of the right hand Milnor square so we must have $a_{1}+\left(J \cup p_{1}\right) B[N]=a_{2}+\left(J \cup \mathfrak{p}_{1}\right) B[N]$. Now we have $a_{1}-a_{2} \in A[M]$ and $a_{1}-a_{2} \in\left(J \cup \mathfrak{p}_{1}\right) B[N]$. But $A[M] \cap\left(J \cup \mathfrak{p}_{1}\right) B[N]=\left(\left(J \cup \mathfrak{p}_{1}\right) \cap M\right) A[M]$ so $a_{1}+\left(\left(J \cup \mathfrak{p}_{1}\right) \cap M\right) A[M]=a_{2}+\left(\left(J \cup \mathfrak{p}_{1}\right) \cap M\right) A[M]$. This means that $a_{1}+(J \cap M) A[M] \in \frac{A[M]}{(J \cap M) A[M]}$ and $a_{2}+\left(\mathfrak{p}_{1} \cap M\right) A[M] \in \frac{A[M]}{\left(\mathfrak{p}_{1} \cap M\right) A[M]}$ have the same image in the lower right hand corner of the left Milnor square. Therefore they patch in the left hand Milnor square to $g=$ $a+(I \cap M) A[M]$ such that $p_{1 \mid}(g)=a+(J \cap M) A[M]=a_{1}+(J \cap M) A[M]$ and $p_{2 \mid}(g)=a+\left(\mathfrak{p}_{1} \cap M\right) A[M]=a_{2}+\left(\mathfrak{p}_{1} \cap M\right) A[M]$. Now we use the inclusions to map this pull back situation to the right Milnor square. This yields $p_{1} i_{I}(g)=i_{J} p_{1\mid}(g)=i_{J}\left(a_{1}+(J \cap M) A[M]\right)=a_{1}+J B[N]=$ $p_{1}(f)$ and $p_{2} i_{I}(g)=i_{\mathfrak{p}_{1}} p_{2 \mid}(g)=i_{\mathfrak{p}_{1}}\left(a_{2}+\left(\mathfrak{p}_{1} \cap M\right) A[M]\right)=a_{2}+\mathfrak{p}_{1} B[N]=$ $p_{2}(f)$. By the unique pullback property of the right Milnor square we have $i_{I}(g)=f$ which implies that $f \in \frac{A[M]}{(I \cap M) A[M]}$ completing the proof.
$\hfill \gj$
\end{proof}

\begin{lemma}\label{Lm3}
Let $A\subset B$ be an extension of rings such that $A$ is subintegrally closed in $B$ and $I$ is a radical ideal of a monoid $M$. If $B$ is reduced, then $\frac{A[M]}{IA[M]}$ is subintegrally closed in $\frac{B[M]}{IB[M]}$ .
\end{lemma}
\begin{proof}
    This is a special case of the Lemma \ref{Lm2} by considering $M=N$.
$\hfill \gj$
\end{proof}

\begin{lemma}\label{LmUnits}
Let $A$ be a reduced ring and $M$ an affine positive monoid. Let $I$ be a radical ideal in $M$. Then $\U(A)=\U\left(\frac{A[M]}{IA[M]}\right)$.
\end{lemma}
\begin{proof}
Since $I$ is a radical ideal, by \cite[Proposition 2.36]{B-G}, we can assume that $I=\p_1\cap \ldots \cap \p_n$ where $\p_i$ are prime ideals in $M$. We prove the result by induction on $n$. If $n=1$, then $I=\p_1$. Then $M'=M \setminus \mathfrak{p_1}$, is a positive submonoid of $M$. Hence  by \cite[Proposition 4.20]{B-G}, we get $\U(A)=\U(A[M'])$. Note that $A[M']=\frac{A[M]}{IA[M]}$, which establishes the case  $n=1$. For $n\geq 2$, let $J=\p_2\cap \ldots \cap \p_n$. Then  the following Milnor squares

\begin{center}

\begin{tikzcd}
A \arrow[d, "\Id"] \arrow[r, "\Id"] & A \arrow[d, "\Id"] &  & {\frac{A[M]}{IA[M]}} \arrow[r, "p_1"] \arrow[d, "p_2"] & {\frac{A[M]}{JA[M]}} \arrow[d, "q_1"']                                            \\
A \arrow[r, "\Id"]                 & A                 &  & {\frac{A[M]}{\mathfrak{p_1}A[M]}} \arrow[r, "q_2"]     & {\frac{A[M]}{\mathfrak({p_1}A[M]+JA[M])}=\frac{A[M]}{(J\cup \mathfrak{p_1})A[M]}}
\end{tikzcd}
\end{center}

yield the following commutative diagram 
\begin{center}
%{\large
\begin{tikzcd}
1 \arrow[r] & \U(A) \arrow[d, "\theta_1"] \arrow[r] & \U(A)\oplus \U(A) \arrow[d, "\theta_2 "] \arrow[r]                   & \U(A) \arrow[d, "\theta_3"] \arrow[r]        &     1                                           \\
1 \arrow[r] & {\U\left(\frac{A[M]}{IA[M]}\right)} \arrow[r]        & {\U\left(\frac{A[M]}{JA[M]}\right)\oplus \U\left(\frac{A[M]}{\mathfrak{p_1}A[M]}\right)} \arrow[r] & {\U\left(\frac{A[M]}{(\mathfrak{p_1}\cup J)A[M]}\right)},
\end{tikzcd}
%}
\end{center}
where $\theta_1$, $\theta_2$ and $\theta_3$
are the group homomorphisms induced by the inclusions
$A\hookrightarrow \frac{A[M]}{IA[M]}$,
$A\hookrightarrow \frac{A[M]}{JA[M]}$,
$A\hookrightarrow \frac{A[M]}{\p_1A[M]}$ and
$A\hookrightarrow \frac{A[M]}{(J\cup \p_1)A[M]}$ respectively. Rows are exact by the Proposition \ref{Bass-Sequence}.

  Now by applying the snake lemma \cite[Lemma 1.3.2]{CW1994Hom}, we have an exact sequence
  $\ker(\theta_1)\rightarrow\ker(\theta_2)\rightarrow\ker(\theta_3)\rightarrow\coker(\theta_1)\rightarrow\coker(\theta_2)\rightarrow\coker(\theta_3)$. By induction $\theta_2$ is an isomorphism, so $\ker(\theta_2)=\coker(\theta_2)=0$. Also $\ker(\theta_1)\rightarrow \ker(\theta_2)$ is an inclusion because $\U(A)\rightarrow \U(A)\oplus \U(A)$ is an inclusion, so $\ker(\theta_1) =0$, i.e., $\theta_1$ is injective. Now we only need to prove $\coker(\theta_1)=0$ and which follows from $\ker(\theta_3) =0$ as $\coker(\theta_2)=0$.
  
  So it remains to prove $\theta_3$ is injective. For this, consider the following sequence of maps $A\xrightarrow{} \frac{A[M]}{(\mathfrak{p}_1\cup J)A[M]} \xrightarrow{} A$, where the first map is an inclusion and the second map is a projection, i.e., induced by the map $A[M]\rightarrow A$, $M\setminus\{1\}\xrightarrow{} 0$, such that composition is the identity map. Since $\U$ is a functor, we get that $\theta_3$ is injective. This completes the proof.
 $\hfill \gj$
\end{proof}

%%%%%%%%%%%%%%%%%%%%%%%%%%%%%%%%%%%%%%%%%%%%%%%%%%%%%%%%%%
\vspace{.3cm}
  Our objective is to prove Theorem \ref{Nil-Thm}, which describes the nilradical $\Nil\left(\frac{A[M]}{IA[M]}\right)$ when $M$ is an affine monoid and  $I$ is a radical ideal in $M$. The proof proceeds by induction on the number of prime components of $I$, making essential use of Lemmas \ref{LmFibProdNill_1} and \ref{LmFibProdNill_2}.

  Let A be a commutative ring and let $I \subseteq J$ be ideals in A. Then the ring homomorphism $f: A / I \rightarrow A / J$ induces a homomorphism of $A$-modules $\left.f\right|_{\operatorname{Nil}(A / I)}: \operatorname{Nil}(A / I) \rightarrow \operatorname{Nil}(A / J)$. By abuse of notation we will write $\left.f\right|_{\mathrm{Nil}(A / I)}=f$.

\begin{lemma}\label{LmFibProdNill_1}
Let $A$ be a commutative ring and $M$ an affine monoid. Let $I=\p_1\cap \ldots \cap \p_n$ be a radical ideal in $M$ and  $J=\p_2\cap \ldots \cap \p_n$, where the $\mathfrak{p_i}$ are prime ideals in $M$. Then the following  Cartesian square
 \begin{equation} \label{FibDiagEqn}
\begin{tikzcd}
{\frac{A[M]}{IA[M]}} \arrow[r,"p_1"] \arrow[d,"p_2"]    & {\frac{A[M]}{JA[M]}} \arrow[d,"q_1"]                                      \\
{\frac{A[M]}{\mathfrak{p_1}A[M]}} \arrow[r,"q_2"] & {\frac{A[M]}{(J\cup \mathfrak{p_1})A[M]}} 
\end{tikzcd}
\end{equation}
induces the following Cartesian square

 \begin{equation}\label{FibDiagEqn*}
\begin{tikzcd}
{\Nil\left(\frac{A[M]}{IA[M]}\right)} \arrow[r,"p_1"] \arrow[d,"p_2"]    & {\Nil\left(\frac{A[M]}{JA[M]}\right)} \arrow[d,"q_1"]                                      \\
{\Nil\left(\frac{A[M]}{\mathfrak{p_1}A[M]}\right)} \arrow[r,"q_2"] & {\Nil\left(\frac{A[M]}{(J\cup \mathfrak{p_1})A[M]}\right)} 
\end{tikzcd}
\end{equation}
in the category of $A[M]$-modules.
\end{lemma}
\begin{proof}
To prove the above square is a Cartesian square we will use Lemma \ref{LmFibProdEqv}. For any pair  $\bar{f}=f+JA[M]\in \Nil\left(\frac{A[M]}{JA[M]}\right)$ and $\bar{g}=g+\mathfrak{p_1}A[M]\in \Nil\left(\frac{A[M]}{\mathfrak{p_1}A[M]}\right)$ such that $q_1(\bar{f})=q_2(\bar{g})$,  we have $f+J A[M]+\mathfrak{p}_{1} A[M]=g+J A[M]+\mathfrak{p}_{1} A[M]$. Then there exists $u, u^{\prime} \in J A[M]$ and $v, v^{\prime} \in \mathfrak{p}_{1}A[M]$ such that $f+u+v=g+u^{\prime}+v^{\prime}$, which implies $f+u-u^{\prime}=g+v^{\prime}-v$. 
Now by taking $\bar{F} = f+u-u' + IA[M]$, we have $p_1(\bar{F}) = f+u-u'+JA[M] = f+JA[M] = \bar{f}$ and $p_2(\bar{F}) = f+u-u'+\mathfrak{p_1}A[M] = g+v'-v + \mathfrak{p_1}A[M] = g+\mathfrak{p_1}A[M]= \bar{g}$. Its remains to show that $\bar{F}\in  \Nil\left(\frac{A[M]}{IA[M]}\right)$. Since $\bar{f}\in \Nil\left(\frac{A[M]}{JA[M]}\right)$ and $\bar{g}\in \Nil\left(\frac{A[M]}{\mathfrak{p_1}A[M]}\right)$, there exists integers $k>0, k'>0$ such that $(f+u-u')^k\in JA[M]$ and $(g+v'-v)^{k'}\in \mathfrak{p_1}A[M]$. Then $(f+u-u')^{k+k'} = (g+v'-v)^{k+k'}\in (J\cap \mathfrak{p_1})A[M] = IA[M]$ and hence $\bar{F}=f+u-u'+IA[M]$ is a  nilpotent element. This $\bar{F}$ is unique because the diagram (\ref{FibDiagEqn}) of Lemma \ref{LmFibProdNill_1} is a Cartesian square. This completes the proof.
 $\hfill \gj$
\end{proof}
\vspace{.3cm}

    Let $A$ be a commutative ring and $M$  a  monoid. Observe that $\Nil(A)[M]$ is an abelian subgroup of $A[M]$. Now, let $f= \sum\limits_{ x\in M}f_{x}x\in A[M]$ and $g= \sum\limits_{ y\in M}g_{y}y\in \Nil(A)[M]$, where $f_{x}$ are in $A$ and $g_{y}$ are in $\Nil(A)$. Then $fg= \sum\limits_{x\in M,y\in M}f_{x}g_{y}xy\in \Nil(A)[M]$ as $f_xg_y$ are nilpotent. Therefore $\Nil(A)[M]$ is an $A[M]$-module.
   For any ideal $I$ in $M$, $I\Nil(A)[M]$ is a submodule of $\Nil(A)[M]$ and hence  $\frac{\Nil(A)[M]}{I\Nil(A)[M]}$ has an $A[M]$-module structure.
   
%\begin{lemma}\label{NilLmIso}[Not nedded]
  %   Let $A$ be a commutative ring and $M$ a finitely generated commutative cancellative torsion-free monoid. Let $I$ be a ideal in $M$. Then $\Nil(A)[\frac{M}{I}]\cong \frac{\Nil(A)[M]}{\Nil(A)[I]}$ as $A[M]$-modules.
%\end{lemma}
%\begin{proof}
%    The monoid $\frac{M}{I}$ is a pointed monoid. Let the base point be $*$. The left hand side $\Nil(A[\frac{M}{I}])$ is finite sums of the form $\sum\limits_{ x\in M}f_{x}x$ with $f_x\in \Nil(A)$ and $x\notin I$, modulo the summand indexed by the base point $*$ of $\frac{M}{I}$. On the right hand side $\Nil(A)[M]$ is finite sums $\sum\limits_{ x\in M}f_{x}x$ with $f_x\in \Nil(A)$ and $\Nil(A)[I]$ is finite sums $\sum\limits_{ x\in I}f_{x}x$ with $f_x\in \Nil(A)$. When we form the quotient  $\frac{\Nil(A)[M]}{\Nil(A)[I]}$ omit those terms with $x\in I$, obtaining the same expression we had for $\Nil(A)[\frac{M}{I}]$.
 %   Hence $\Nil(A)[\frac{M}{I}]\cong \frac{\Nil(A)[M]}{\Nil(A)[I]}$.
 %    $\hfill \gj$
%\end{proof}

\begin{lemma}\label{LmFibProdNill_2}
Let $A$ be a commutative ring and $M$ an affine monoid. Let $I=\p_1\cap \ldots \cap \p_n$ be a radical ideal in $M$ and  $J=\p_2\cap \ldots \cap \p_n$, where the $\mathfrak{p_i}$ are prime ideals in $M$. 
Then the following commutative diagram 
\begin{equation} \label{Diag_LmFibProdNill_2}
\begin{tikzcd}
\frac{\Nil(A)[M]}{I\Nil(A)[M]} \arrow[r, "p'_1"] \arrow[d, "p'_2"]    & \frac{\Nil(A)[M]}{J\Nil(A)[M]} \arrow[d, "q'_1"]                                      \\
\frac{\Nil(A)[M]}{\p_1\Nil(A)[M]} \arrow[r, "q'_2"] & \frac{\Nil(A)[M]}{(J\cup \p_1)\Nil(A)[M]}
\end{tikzcd}
\end{equation}
is a Cartesian square in the category of $A[M]$-modules, where $p'_1$ is an $A[M]$-module homomorphism given by $p'_1(f+I\Nil(A)[M]) = f + J\Nil(A)[M]$ and is well-defined because $I\subset J$. Similarly one can define $A[M]$-module homomorphisms  $p'_2$, $q'_1$, $q'_2$.
\end{lemma}
\begin{proof}
Let $\bar{f}\in \frac{\Nil(A)[M]}{J\Nil(A)[M]}$ and $\bar{g}\in \frac{\Nil(A)[M]}{p_1\Nil(A)[M]}$
  such that $q'_1(\bar{f})=q'_2(\bar{g})$. 
 Then $\bar{f}=f+J \Nil(A)[M]$ and $\bar{g}=g+$ $\mathfrak{p}_{1} \Nil(A)[M]$ for some $f, g \in \Nil(A)[M]$. Thus $f+\left(J \cup \mathfrak{p}_{1}\right) \Nil(A)[M]=$ $g+\left(J \cup \mathfrak{p}_{1}\right) \Nil(A)[M]$. Because $(J\cup\p_1)\Nil(A)[M] = J\Nil(A)[M]+\p_1\Nil(A)[M]$, there exists $u\in J\Nil(A)[M]$ and $v\in p_1\Nil(A)[M]$ such that $f+u=g+v$. Now by taking $\bar{F} = f+u + I\Nil(A)[M]$, we have $p'_1(\bar{F}) = f+u+J\Nil(A)[M] = f+J\Nil(A)[M] = \bar{f}$ and $p'_2(\bar{F}) = f+u+\p_1\Nil(A)[M] = g+v + \p_1\Nil(A)[M] = g+\p_1\Nil(A)[M]= \bar{g}$.
 This $\bar{F}$ is unique because the diagram $(\ref{Diag_LmFibProdNill_2})$ is a sub-diagram of the Cartesian square $(\ref{FibDiagEqn})$ of Lemma \ref{LmFibProdNill_1}. Hence by  Lemma \ref{LmFibProdEqv}, the above square is a Cartesian square in the category of $A[M]$-modules. 
$\hfill \gj$
\end{proof}

 \begin{theorem}\label{Nil-Thm}
 Let $A$ be a commutative ring and $M$ an affine monoid. Let $I$ be a radical ideal in $M$. Then $\Nil\left(\frac{A[M]}{IA[M]}\right) \cong \frac{\Nil(A)[M]}{I\Nil(A)[M]}$ as $A[M]$-modules.
 \end{theorem}
 \begin{proof}
 Since $I$ is a radical ideal, by \cite[Proposition 2.36]{B-G}, we can assume that $I=\p_1\cap \ldots \cap \p_n$,  where $\mathfrak{p_i}$ are prime ideals in $M$. We prove the result by induction on $n$. If $n=1$ then $I=\mathfrak{p}_{1}$. Since $I$ is a prime ideal, $M^{\prime}=M \backslash \p_1$ is a submonoid of $M$ and the $A[M] \rightarrow A[M']$ that sends $\sum\limits_{m_i\in M} a_{i} m_{i}$ to $\sum\limits_{m_i\notin \p_1} a_{i} m_{i}$ with $a_{i}\in A$ is an $A[M]$-algebra homomorphism with kernel $\mathfrak{p}_{1} A[M]$. Therefore $A[M'] \cong \frac{A[M]}{\mathfrak{p}_{1} A[M]}$. Similarly $\Nil(A)[M'] \cong \frac{\Nil(A)[M]}{\mathfrak{p}_{1} \Nil(A)[M]}$. Thus it is enough to show that  $\Nil(A[M'])=\Nil(A)[M']$. Elements of $\Nil(A)\left[M^{\prime}\right]$ are finite sums $\Sigma a_{i} m_{i}$ with $a_{i}$ a nilpotent element of $A$ and $m_{i} \notin \mathfrak{p}_{1}$. Such elements are obviously nilpotent so we have an inclusion of $A$-modules $\Nil(A)\left[M^{\prime}\right] \subset$ $\Nil\left(A\left[M^{\prime}\right]\right)$, which we follow by $\Nil\left(A\left[M^{\prime}\right]\right) \subset A\left[M^{\prime}\right]$.
The surjection $A\left[M^{\prime}\right] \twoheadrightarrow (A / \Nil(A))\left[M^{\prime}\right]$ with kernel $\Nil(A)[M']$ induces the isomorphism
$A\left[M^{\prime}\right] / \Nil(A)\left[M^{\prime}\right] \cong
(A / \Nil(A))\left[M^{\prime}\right]$
. The last ring is reduced by \cite[Theorem 4.19]{B-G}, from which it follows that $\Nil\left(A\left[M^{\prime}\right]\right) \subset$ $\Nil(A)\left[M^{\prime}\right]$. Hence we get $\Nil(A[M'])=\Nil(A)[M']$ which establishes the case $n=1$.
 %Hence  by \cite[Proposition 4.19]{B-G}, we get $\Nil(A[M'])=\Nil(A)[M']$ which establishes the case $n=1$.

   For $n\geq 2$, let $J=\p_2\cap \ldots \cap \p_n$. Consider  the following Cartesian square 
\begin{center}
\begin{tikzcd}
{\frac{A[M]}{IA[M]}} \arrow[r,"p_1"] \arrow[d,"p_2"]    & {\frac{A[M]}{JA[M]}} \arrow[d,"q_1"]                                      \\
{\frac{A[M]}{\mathfrak{p_1}A[M]}} \arrow[r,"q_2"] & {\frac{A[M]}{(J\cup \mathfrak{p_1})A[M]}} 
\end{tikzcd}
\end{center}
in the category of rings, where all  $p_1, p_2, q_1, q_2$ are natural quotient maps.
 $\Nil : \textbf{Rings}\longrightarrow \textbf{AbGr}$ is a functor.
  After applying $\Nil$ and  by Lemma \ref{LmFibProdNill_1} the following diagram is a Cartesian square
 \begin{center}
\begin{tikzcd}
{\Nil\left(\frac{A[M]}{IA[M]}\right)} \arrow[r,"p_1"] \arrow[d,"p_2"]    & {\Nil\left(\frac{A[M]}{JA[M]}\right)} \arrow[d,"q_1"]                                      \\
{\Nil\left(\frac{A[M]}{\mathfrak{p_1}A[M]}\right)} \arrow[r,"q_2"] & {\Nil\left(\frac{A[M]}{(J\cup \mathfrak{p_1})A[M]}\right)} 
\end{tikzcd}
\end{center}
 in the category of  $A[M]$-modules, where all maps are restricted to the corresponding \Nil ideals of natural quotient maps  $p_1, p_2, q_1,q_2$ and then again we use the same notations. Then we have the following exact sequence
 \[
       \xymatrix{
        0\ar@{->}[rr]&& \Nil\left(\frac{A[M]}{IA[M]}\right)\ar@{->}[rr]^(.35){(p_1,p_2)} & & \Nil\left(\frac{A[M]}{JA[M]}\right)\op \Nil\left(\frac{A[M]}{\p_1 A[M]}\right)\ar@{->}[rr]^(.6){q_1-q_2} & & \Nil\left(\frac{A[M]}{(J\cup\p_1)A[M]}\right)
       }.
      \] 
Again by Lemma \ref{LmFibProdNill_2}, we have the following Cartesian square
\begin{center}
\begin{tikzcd}
\frac{\Nil(A)[M]}{I\Nil(A)[M]} \arrow[r, "p'_1"] \arrow[d, "p'_2"]    & \frac{\Nil(A)[M]}{J\Nil(A)[M]} \arrow[d, "q'_1"]                                      \\
\frac{\Nil(A)[M]}{\p_1\Nil(A)[M]} \arrow[r, "q'_2"] & \frac{\Nil(A)[M]}{(J\cup \p_1)\Nil(A)[M]}
\end{tikzcd}
\end{center}
in the category of $A[M]$-modules. In this diagram $p_{1}'$ is defined by $p_{1}'(f+I \Nil(A)[M])=f+J \Nil(A)[M]$ (which is well-defined because $I \subset J)$. Similarly one defines the other maps $p_{2}', q_{1}', q_{2}'$. Note that the maps $q_{1}', q_{2}'$ are onto, and hence $q_{1}'-q_{2}'$ is onto.

Then we have the following exact sequence
\[
       \xymatrix{
        0\ar@{->}[rr]&&\frac{\Nil(A)[M]}{I\Nil(A)[M]}\ar@{->}[rr]^(.35){(p'_1,p'_2)} & & \frac{\Nil(A)[M]}{J\Nil(A)[M]}\op \frac{\Nil(A)[M]}{\p_1\Nil(A)[M]}\ar@{->}[rr]^(.6){q'_1-q'_2} & & \frac{\Nil(A)[M]}{(J\cup \p_1)\Nil(A)[M]}\ar@{->}[rr]&& 0
       }.
      \] 
Combining the above two exact sequences, we get the following 
\begin{equation}\label{EqNilDiagSnake1}
\begin{tikzcd}
0 \arrow[r] & \frac{\Nil(A)[M]}{I\Nil(A)[M]} \arrow[d, "\theta_1"] \arrow[r, "{(p'_1, \ p'_2)}"] & \frac{\Nil(A)[M]}{J\Nil(A)[M]}\oplus \frac{\Nil(A)[M]}{\p_1\Nil(A)[M]} \arrow[d, "\theta_2"] \arrow[r, "q'_1-q'_2"] & \frac{\Nil(A)[M]}{(J\cup \p_1)\Nil(A)[M]}\arrow[d, "\theta_3"] \arrow[r, " "]& 0\\
0 \arrow[r] & \Nil\left(\frac{A[M]}{IA[M]}\right) \arrow[r, "{(p_1, \ p_2)}"]                       & \Nil\left(\frac{A[M]}{JA[M]}\right)\oplus \Nil\left(\frac{A[M]}{\p_1A[M]}\right) \arrow[r, "q_1-q_2"]                              & \Nil\left(\frac{A[M]}{(J\cup \p_1)A[M]}\right)                             
\end{tikzcd}
\end{equation}
commutative diagram with exact rows. The maps $\theta_{i}$ are defined as follows. Every element of $\frac{\Nil(A)[M]}{I\Nil(A)[M]}$ is of the form $F= \sum\limits_{ x\in M}f_{x}x +I\Nil(A)[M]$ with finitely many $f_x\in \Nil(A)$ non-zero. Since $I\Nil(A)[M]\subset IA[M]$, $F$ also represents a nilpotent element of $\frac{A[M]}{IA[M]}$. Thus we have the inclusion  $\theta_1: \frac{\Nil(A)[M]}{I\Nil(A)[M]} \rightarrow \Nil\left(\frac{A[M]}{IA[M]}\right)$. Replacing $I$ by $J, \mathfrak{p}_{1}$ and $J \cup \mathfrak{p}_{1}$ we also obtain inclusions $\theta_{2}$ and $\theta_{3}$.

Now by applying the snake lemma \cite[Lemma 1.3.2]{CW1994Hom} to the diagram $(\ref{EqNilDiagSnake1})$, we get the exact sequence $\ker(\theta_1)\rightarrow\ker(\theta_2)\rightarrow\ker(\theta_3)\rightarrow\coker(\theta_1)\rightarrow\coker(\theta_2)\rightarrow\coker(\theta_3)$. Since $\theta_{1}, \theta_{2}, \theta_{3}$ are inclusions we have $\ker(\theta_{1})=\ker(\theta_{2})=\ker(\theta_{3})=0$. By induction $\theta_{2}$ is an isomorphism so we also have $\coker(\theta_{2})=0$. Thus the snake lemma simplifies to $0 \rightarrow \coker(\theta_{1}) \rightarrow 0$ making $\theta_{1}$ an isomorphism, which completes the proof.
 $\hfill \gj$
 \end{proof}

\begin{corollary}
    Let $A$ be a commutative ring and $M$ an affine monoid. Let $I$ be a radical ideal in $M$. If $\Nil(A) = 0$, then $\Nil\left(\frac{A[M]}{IA[M]}\right)= 0$.
\end{corollary}
\begin{proof}
Follows form Theorem \ref{Nil-Thm}.
 $\hfill \gj$
\end{proof}

 \begin{corollary}\label{Coro-Nil}
Let $A \subset B$ be an extension of commutative rings with $\Nil(A)=\Nil(B)$. Let $M\subset N$ be an extension of affine monoids. Furthermore, assume that $I$ is a radical ideal in $N$. Then $\Nil\left(\frac{A[M]}{(I \cap M) A[M]}\right) \cong \Nil\left(\frac{B[N]}{I B[N]}\right)$ if and only if $B$ is reduced or $M=N$.
 
 \end{corollary}
 \begin{proof}
$(:\Rightarrow)$ Let $\Nil\left(\frac{A[M]}{(I \cap M) A[M]}\right) \cong \Nil\left(\frac{B[N]}{I B[N]}\right)$ holds true for a radical ideal $I$ in $N$. Suppose that $\Nil(B)\neq 0$, by Theorem \ref{Nil-Thm} we have $\Nil\left(\frac{A[M]}{(I \cap M) A[M]}\right) \cong \frac{\Nil(A)[M]}{(I \cap M) \Nil(A)[M]}$ and  $\Nil\left(\frac{B[N]}{I B[N]}\right)\cong\frac{\Nil(B)[N]}{I\Nil(B)[N]}$. Since $\Nil(A)=\Nil(B)$ we have $\frac{\Nil(A)[M]}{(I \cap M) \Nil(A)[M]} \cong \frac{\Nil(B)[M]}{(I \cap M) \Nil(B)[M]}$. Thus by assumption  $\frac{\Nil(B)[M]}{(I\cap M)\Nil(B)[M]} \cong \frac{\Nil(B)[N]}{I \Nil(B)[N]}$, but this is not true unless $M=N$. Indeed, if $M \subsetneq N$. We have the natural map $f : \frac{\Nil(B)[M]}{(I \cap M) \Nil(B)[M]} \rightarrow \frac{\Nil(B)[N]}{I \Nil(B)[N]}$ induced by $M \subset N$. Elements of the image of $f$ are uniquely sums $\sum\limits_{n_{i}\in N\backslash I}b_{i}n_{i}$ with $b_{i}\in \Nil(B)$ and those in the domain of $f$ are uniquely sums $\sum\limits_{n_{i}\in M\backslash I}b_{i}n_{i}= \sum\limits_{n_{i}\in M\backslash (I\cap M)}b_{i}n_{i}$ with $b_{i}\in \Nil(B)$ and if $M \subsetneq N$ there are more summands in the image so $f$ is not onto and thus not isomorphism. Hence we must have $M=N$.

$(\Leftarrow:)$ Suppose either  $\Nil(B)=0$ or $M=N$. After using Theorem \ref{Nil-Thm}  we have  $\Nil\left(\frac{A[M]}{(I \cap M) A[M]}\right) \cong \frac{\Nil(A)[M]}{(I \cap M) \Nil(A)[M]}$ and $\Nil\left(\frac{B[N]}{I B[N]}\right)\cong\frac{\Nil(B)[N]}{I \Nil(B)[N]}$. Since $\Nil(A)=\Nil(B)$ we have $\Nil\left(\frac{A[M]}{(I \cap M) A[M]}\right) \cong \frac{\Nil(B)[M]}{(I \cap M) \Nil(B)[M]}$. Thus to show $\Nil\left(\frac{A[M]}{(I \cap M) A[M]}\right) \cong \Nil\left(\frac{B[N]}{I B[N]}\right)$ it is equivalent to show that $\frac{\Nil(B)[M]}{(I\cap M)\Nil(B)[M]} \cong \frac{\Nil(B)[N]}{I \Nil(B)[N]}$, but the last isomorphism holds true if either $\Nil(B)=0$ or $M=N$.
 $\hfill \gj$
 \end{proof}

 \begin{lemma}\label{Ker-Main Theorem}
     Let $B$ be a commutative ring and $M\subset N$ an extension of affine monoids. Let $I$ be an ideal in $N$. Then the group $\frac{1+\frac{\Nil(B)[N]}{I\Nil(B)[N]}}{1+\frac{\Nil(B)[M]}{(I\cap M)\Nil(B)[M]}}$ is a trivial group if and only if $B$ is reduced or $M=N$.
     
    % $\frac{\Nil(B)[M]}{(I\cap M)\Nil(B)[M]}  \cong \frac{\Nil(B)[N]}{I\Nil(B)[N]}$.
 \end{lemma}
 \begin{proof}
  Let the group $\frac{1+\frac{\Nil(B)[N]}{I\Nil(B)[N]}}{1+\frac{\Nil(B)[M]}{(I\cap M)\Nil(B)[M]}}$ be trivial. Then either $\Nil(B) = 0$ or $\frac{\Nil(B)[M]}{(I\cap M)\Nil(B)[M]}  \cong \frac{\Nil(B)[N]}{I\Nil(B)[N]}$. Applying Theorem \ref{Nil-Thm} and then by Corollary \ref{Coro-Nil} we get either $\Nil(B)=0$ or $M=N$.
  \par
  Conversely, assume that $\Nil(B)=0$ or $M=N$. Then in either case the group $\frac{1+\frac{\Nil(B)[N]}{I\Nil(B)[N]}}{1+\frac{\Nil(B)[M]}{(I\cap M)\Nil(B)[M]}}$ is trivial. 
 $\hfill \gj$
 \end{proof}

\subsection{Main Theorem}

The following result due to Sarwar \cite[Theorem 1.2]{Sarwar} is essential to our subsequent theorem.
 \begin{theorem}\cite[Theorem $1.2$]{Sarwar}
  \label{Sarwar-Main-Thm}
  Let $A\subset B$ be an extension of rings and $M \subset N$ an extension of positive monoids.

  \begin{enumerate}
      \item[(a)]
    If $A[M]$ is subintegrally closed in $B[N]$ and $N$ is affine, then $\mathcal{I}(A, B) \cong \mathcal{I}(A[M], B[N])$.
      \item[(b)]
    If $B$ is reduced, $A$ is subintegrally closed in $B$ and $M$ is subintegrally closed in $N$, then $\mathcal{I}(A, B) \cong \mathcal{I}(A[M], B[N])$.

      \item[(c)] 
    If $M = N$, then the reduced condition on $B$ is not needed i.e. if $A$ is subintegrally closed in $B$, then $\mathcal{I}(A, B) \cong \mathcal{I}(A[M], B[M])$.

      \item[(d)] 
    (Converse of (a),(b) and (c)) If $\mathcal{I}(A, B) \cong \mathcal{I}(A[M], B[N])$, then (i) $A$ is subintegrally closed in
$B$, (ii) $A[M]$ is subintegrally closed in $B[N]$, and (iii) $B$ is reduced or $M = N$.
  \end{enumerate}
\end{theorem}

\begin{theorem}\label{MainTheorem-1}
Let $A\subset B$ be an extension of rings and $M\subset N$ an extension of positive monoids. Let $I$ be a radical ideal in $N$.
  \begin{enumerate}
  \item[(a)]
   If $\frac{A[M]}{(I\cap M)A[M]}$ is subintegrally closed in $\frac{B[N]}{IB[N]}$ and $N$ is affine, then $\mathcal{I}(A, B)\cong \mathcal{I}\left(\frac{A[M]}{(I\cap M)A[M]}, \frac{B[N]}{IB[N]}\right)$.
   \item[(b)]
   If $B$ is reduced, $A$ is subintegrally closed in $B$, and $M$ is subintegrally closed in $N$, then $\mathcal{I}(A, B)\cong \mathcal{I}\left(\frac{A[M]}{(I\cap M)A[M]}, \frac{B[N]}{IB[N]}\right)$.
   \item[(c)]
   If $M=N$, and $A$ is subintegrally closed in $B$, then $\mathcal{I}(A, B)\cong \mathcal{I}\left(\frac{A[M]}{IA[M]}, \frac{B[M]}{IB[M]}\right)$.
   \item[(d)]
   If $\mathcal{I}(A, B)\cong \mathcal{I}\left(\frac{A[M]}{(I\cap M)A[M]}, \frac{B[N]}{IB[N]}\right)$ and $N$ is affine, then (i) $A$ is subintegrally closed in $B$, (ii) $\frac{A[M]}{(I\cap M)A[M]}$ is subintegrally closed in $\frac{B[N]}{IB[N]}$, (iii) $B$ is reduced or $M=N$.
\end{enumerate}
\end{theorem}

\begin{proof}
 (a)  Since $I$ is a radical ideal,
 by \cite[Proposition 2.36(a)]{B-G}, we can assume that $I=\p_1\cap \ldots \cap \p_n$ where $\mathfrak{p_i}$ are prime ideals in $N$. We prove the result by induction on $n$. If $n=1$, then $I=\p_1$. Hence by \cite[Proposition 2.36(b)]{B-G}, $I=N\setminus (F\cap N)$, where $F$ is a face of $\BR_+N$ (here we are writing $N$ additively). 
 %%%%%%%%%%%%%%%%
    Now we have $\frac{B[N]}{I B[N]} \cong$ $B[N \cap F]$ because both sides have $B$-basis those elements of $N$ not in $I$. The map  $\frac{A[M]}{(I \cap M) A[M]} \hookrightarrow \frac{B[N]}{I B[N]}$ is an inclusion because $A \subset B$ and the $A$-basis of the left hand side is a subset of the $A$-basis of the right hand side. We also have an isomorphism $\frac{A[M]}{(I \cap M) A[M]} \cong A[M \cap F]$ because again both sides are the free $A$-module with basis those elements $n \in M$ that are not contained in $I$. The hypothesis says that $A[M \cap F]$ is subintegrally closed in $B[N \cap$ $F]$. So apply Theorem \ref{Sarwar-Main-Thm}(a) with $A \subset B$ and $M \cap$ $F \subset N \cap F$ an extension of positive monoids and $N \cap F$ affine. This yields $\mathcal{I}(A, B) \cong \mathcal{I}(A[M \cap F], B[N \cap F])$. This establishes the case $n=1$.

%Therefore by \cite[Corollary 4.34]{B-G}, $B[N\cap F]\cong B[N]/IB[N]$, $A[M\cap F]\cong A[M]/(I\cap F)A[M]$. Hence by \cite[Theorem 1.2(a)]{Sarwar} $\CI(A,B)\cong \CI(A[M\cap F],B[N\cap F])$ (notice here that $A$ is subintegrally closed in $B$ and $M\cap F$ is subintegrally closed in $N\cap F$). So the basic step is done.
 
 Assume that $n>1$.  Let $J=\p_2\cap \ldots \cap \p_n$. Note that $JB[N]+\p_1B[N] = (J\cup \p_1)B[N]$.
 Consider the following Milnor square in the category of ring extensions

\begin{equation}\label{MainMilnorsquare}
    \begin{tikzcd}
\frac{A[M]}{(I\cap M)A[M]}\subset  \frac{B[N]}{IB[N]} \ar[r] \arrow[d, "p_2"] \arrow[r, "p_1"] & \frac{A[M]}{(J\cap M)A[M]}\subset \frac{B[N]}{JB[N]} \arrow[d, "q_1"]  \\
\frac{A[M]}{(\p_1\cap M)A[M]} \subset  \frac{B[ N]}{\p_1B[N]} \arrow[r, "q_2"]              & \frac{A[M]}{((J\cup \p_1)\cap M)A[M]} \subset  \frac{B[N]}{(J\cup \p_1)B[N]}           \arrow[u, shift left=.9ex, "q'_1"]       
\end{tikzcd}
\end{equation}

Consider  the submonoid $N'=N\setminus \mathfrak{p_1}$ and observe that  
\begin{equation}\label{EqnInsideMain}
    \frac{B[N]}{(J\cup \mathfrak{p_1})B[N]}=\frac{B[N']}{(N'\cap J)B[N']}.
\end{equation}
Then $N'\subset N$ induces the following $B$-algebra homomorphisms $B[N']\rightarrow B[N]\rightarrow \frac{B[N]}{JB[N]}$. The  kernel of this composition of $B$-algebras is the ideal $(J\cap N')B[N']$ in $B[N']$, so we have an injective $B$-algebra homomorphism $\frac{B[N']}{(J\cap N')B[N']} \xrightarrow{} \frac{B[N]}{JB[N]}$.
 %which further induce an injective $B$-algebra homomorphism $\frac{B[N']}{(J\cap N')B[N']} \xrightarrow{} \frac{B[N]}{JB[N]}$. 
By the identification in equation (\ref{EqnInsideMain}), we get an injective $B$-algebra homomorphism $q'_1: \frac{B[N]}{(J\cup \mathfrak{p_1})B[N]} \xrightarrow{} \frac{B[N]}{JB[N]}$ which is a retraction of the natural quotient map $q_1: \frac{B[N]}{JB[N]} \xrightarrow{} \frac{B[N]}{(J\cup \mathfrak{p_1})B[N]}$, i.e., $q_1\circ q'_1 = \Id$. Also, we have $q_1\circ q'_1 = \Id$ when we restrict to $A[M]/(((J\cup \p_1)\cap M)A[M])$. Thus $q'_1$ is the retraction of $q_1$, which is a homomorphism of ring extensions.

 Applying Lemma \ref{7l1} to the Milnor square (\ref{MainMilnorsquare}) we have the following exact sequence
  {\tiny
 \[
       \xymatrix{
        1\ar@{->}[r] &  \CI\left(\frac{A[M]}{(I\cap M)A[M]},\frac{B[N]}{IB[N]}\right)\ar@{->}[r]^(.35){\phi}  & \CI\left(\frac{A[M]}{(J\cap M)A[M]},\frac{B[N]}{JB[N]}\right)\op\CI\left(\frac{A[M]}{(\p_1\cap M)A[M]},\frac{B[N]}{\p_1B[N]}\right)\ar@{->}[r]^(.6){\psi}  &\CI\left(\frac{A[M]}{((J\cup \p_1)\cap M)A[M]},\frac{B[N]}{(J\cup \p_1)B[N]}\right),
       }
      \] }   
   where maps $\phi$ and $\psi$ as defined in Lemma~\ref{7l1}.
\par
  Let us consider the following commutative diagram  with exact rows 
 {\tiny
 \[
       \xymatrix{
       1\ar@{->}[r]  & \CI(A,B) \ar@{->}[r]^{\phi_1}\ar@{->}[d]^{\gt_1} &
       \CI(A,B)\op \CI(A,B)\ar@{->}[r]^{\psi_1}\ar@{->}[d]^{\gt_2} & \CI(A,B)\ar@{->}[d]^{\gt_3} &\\
      1\ar@{->}[r]  & \CI\left(\frac{A[M]}{(I\cap M)A[M]},\frac{B[N]}{IB[N]}\right)\ar@{->}[r]^(.34){\phi}& \CI\left(\frac{A[M]}{(J\cap M)A[M]},\frac{B[N]}{JB[N]}\right)\op
      \CI\left(\frac{A[M]}{(\p_1\cap M)A[M]},\frac{B[N]}{\p_1B[N]}\right)\ar@{->}[r]^(.6){\psi}& \CI\left(\frac{A[M]}{((J\cup \p_1)\cap M)A[M]},\frac{B[N]}{(J\cup \p_1)B[N]}\right),
       }
      \]
      }
where maps $\phi_1$, $\phi$, $\psi_1$, $\psi$  as defined in Lemma \ref{7l1}. The map $\theta_1$ is induced by the following morphism of ring extension $(A,B)\longrightarrow \left(\frac{A[M]}{(I\cap M)A[M]},\frac{B[N]}{IB[N]}\right)$,  $\theta_2$ is induced by the following morphisms of ring extensions $(A,B)\longrightarrow \left(\frac{A[M]}{(J\cap M)A[M]},\frac{B[N]}{JB[N]}\right)$ and  $(A,B)\longrightarrow \left(\frac{A[M]}{(\p_1\cap M)A[M]},\frac{B[N]}{\p_1B[N]}\right)$ and $\theta_3$ is induced by the following morphism of ring extension $(A,B)\longrightarrow \left(\frac{A[M]}{((J\cup\p_1)\cap M)A[M]},\frac{B[N]}{(J\cup \p_1)B[N]}\right)$. 
      
By the induction hypothesis, $\gt_2$ is an isomorphism. We have to prove $\theta_1$ is an isomorphism. First, observe that the composition of the maps $(A,B)\longrightarrow \left(\frac{A[M]}{(I\cap M)A[M]},\frac{B[N]}{IB[N]}\right)\longrightarrow (A,B)$ is the identity, where the first map is inclusion and the second map is the natural projection $(N\xrightarrow{} 1)$. Now applying the functor $\mathcal{I}$, we get that the map $\theta_1$ is injective. Similarly the map $\gt_3$ is also injective.

To prove $\gt_1$ is surjective, let $I'\in \mathcal{I}\left(\frac{A[M]}{(I\cap M)A[M]},\frac{B[N]}{IB[N]}\right)$, and let $\phi(I') = I_{1}$. Because $\gt_2$ is an isomorphism, there exists $I_{2}\in \CI(A,B)\oplus \CI(A,B)$ such that $\gt_2(I_{2}) = I_{1}$. By commutativity of the upper right square we have $\gt_3\psi_1(I_2) = \psi \gt_2(I_2)$. Hence $\gt_3 \psi_1(I_2) = \psi\gt_2(I_2) = \psi(I_1) = \psi\phi(I') = 1$. Since $\gt_3$ is injective we conclude that $\psi_1(I_2) = 1$. By exactness of the first row
there exists $I_3\in \CI(A,B)$ such that $ \phi_1(I_3) = I_2$.  By commutativity of the  square  with top edge $\phi_1$ we have 
 \[
 \begin{split}
\phi(I'-\gt_1(I_3)) &=\phi(I')-\phi\gt_1(I_3)\\ 
  &= \phi(I')-\gt_2\phi_1(I_3)\\
   &=\phi(I')-\gt_2(I_2)\\
   &= I_1-I_1\\
   &= 0.
 \end{split}
 \]
Since $\phi$ is injective, we have $\gt_{1}(I_3) = I'$. Hence, the map $\theta_1$ is surjective. This completes the proof that $\theta_1$ is an isomorphism.

\vspace{.3cm}
\par
(b) First, we assume $N$ is affine. Since $B$ is reduced by Lemma \ref{Lm2}, we get $\frac{A[M]}{(I\cap M)A[M]}$ is subintegrally closed in $\frac{B[N]}{IB[N]}$. Hence by part (a), we have $\mathcal{I}(A, B)\cong \mathcal{I}\left(\frac{A[M]}{(I\cap M)A[M]}, \frac{B[N])}{IB[N]}\right)$.\\
 If $N$ is not affine, we can write $N$ as the filtered union of 
 affine monoids. Then, this part can be proved in a similar way as in the proof of  \cite[Theorem 1.2(b)]{Sarwar}. We leave out the details for the reader.
\vspace{.3cm}
\par
(c) Since $A$ is subintegrally closed in $B$,  by Lemma \ref{Lm3}, we have
$\frac{A[M]}{IA[M]}$ is subintegrally closed in $\frac{B[M]}{IB[M]}$.
Hence by part (a),
we have $\mathcal{I}(A, B)\cong \mathcal{I}\left(\frac{A[M]}{IA[M]}, \frac{B[M]}{IA[M]}\right)$ for the case when $M$ is affine. If $M$ is not affine, then we can write $M$ as a filtered union of affine monoids and proceed as in part (b).
\vspace{.3cm}
\par
(d)(i) To prove $A$ is subintegrally closed in $B$, let $b\in B$ with $b^2, b^3\in A$. Let $m\in M\setminus I$. Let $J:= (\overline{b}^2, 1-\overline{bm})$ and $J':= (\overline{b}^2, 1+\overline{bm})$ be two $\frac{A[M]}{(I\cap M)A[M]}$-submodule of $\frac{B[N]}{IB[N]}$. It is clear that $JJ'\subset \frac{A[M]}{(I\cap M)A[M]}$ and $(1-\overline{bm})(1+\overline{bm})(1+\overline{b^2m^2})=1-\overline{b^4m^4}\in JJ'$. Hence $1=\overline{b^4m^4}+1-\overline{b^4m^4}\in JJ'$ i.e. $JJ'=\frac{A[M]}{(I\cap M)A[M]}$. Consequently, $J\in \mathcal{I}\left(\frac{A[M]}{(I\cap M)A[M]}, \frac{B[N]}{IB[N]}\right)$. Let $\pi$ be the natural surjection from $\frac{B[N]}{IB[N]}$ to $B$ induced by the surjection $B[N]\rightarrow B$ sending $N\setminus \{1\}\rightarrow0$. Then $\CI(\pi)(J)=A$. By hypothesis $\CI(\pi)$ is an isomorphism, hence $J=\frac{A[M]}{(I\cap M)A[M]}$. Therefore $b\in A$. Hence $A$ is subintegrally closed in $B$.

%(d) (i) One can prove it in similar lines as in part (d)(i) of Theorem \ref{Sarwar-Main-Thm}.\\
 %(ii)  One can prove it in similar lines as in previous part (d)(i) of Theorem \ref{MainTheorem-1}.
 \vspace{.3cm}
\par
 (d)(ii) One can prove it in similar lines as in part (d)(ii) of Theorem \ref{Sarwar-Main-Thm}. However we provide a proof for the shake of completeness. Let $\overline{g}\in \frac{B[N]}{IB[N]}$ such that $\overline{g}^2, \overline{g}^3\in \frac{A[M]}{(I\cap M)A[M]}$. Let $J:=(\overline{g}^2, 1+\overline{g}+\overline{g}^2)$ and Let $J':=(\overline{g}^2, 1-\overline{g}+\overline{g}^2)$ be two $\frac{A[M]}{(I\cap M)A[M]}$-submodule of $\frac{B[N]}{IB[N]}$.  It is clear that $JJ'\subset \mathcal{I}(\frac{A[M]}{(I\cap M)A[M]}$. Then $(1+\overline{g}+\overline{g}^2)(1-\overline{g}+\overline{g}^2)=1+\overline{g}^2+\overline{g}^4)\in JJ'$, which implies $1+\overline{g}^2\in JJ'$. Thus we have $1=(1+\overline{g}^2)(1-\overline{g}^2)+\overline{g}^4\in JJ'$, and hence $JJ'=\frac{A[M]}{(I\cap M)A[M]}$.
Consequently, $J\in \mathcal{I}(\frac{A[M]}{(I\cap M)A[M]}, \frac{B[N]}{IB[N]})$. Let $\pi$ be the natural surjection from $\frac{B[N]}{IB[N]}$ to $B$ induced by the surjection $B[N]\rightarrow B$ sending $N\setminus \{1\}\rightarrow0$. Set $\pi(\overline{g})=b\in B$. Then $\CI(\pi)(J)=(b^2, 1+b+b^2)$. Since $\overline{g}^2, \overline{g}^3\in \frac{A[M]}{(I\cap M)A[M]}$, we must have $b^2, b^3\in A$. But by (i) $A$ is subintegrally closed in $B$, thus we get $b\in A$. Hence $\CI(\pi)(J')$ and $\CI(\pi)(J')$ both contained in $A$. Therefore $\CI(\pi)(J)=A$, which implies $J=\frac{A[M]}{(I\cap M)A[M]}$ and hence $\overline{g}\in \frac{A[M]}{(I\cap M)A[M]}$.
This proves that $\frac{A[M]}{(I\cap M)A[M]}$
is subintegrally closed in $\frac{B[N]}{IB[N]}$.

\vspace{.5cm}
%%%%%%%%%%%%%%%%%%%%%%%%%%%%%%%%%%%%%%%%%%%%%%%%%%%%%%%%%%%%%%%%%%%%%%%%%%%%%%
To prove the part $(d)(iii)$, we need the following lemma.
\begin{lemma} \label{Lm_Bar}
Let $A\subset B$ be an extension of rings and $M\subset N$ an extension of cancellative torsion-free positive monoids with $N$ is affine. Let $I$ be a radical ideal in $N$, $\overline{A}=A/\Nil(A)$, and $\overline{B}=B/\Nil(B)$. If $\frac{A[M]}{(I\cap M)A[M]}$ is subintegrally closed in $\frac{B[N]}{IB[N]}$, then $\frac{\overline{A}[M]}{(I \cap M)\overline{A}[M]}$ is subintegrally closed in $\frac{\overline{B}[N]}{I\overline{B}[N]}$.
\end{lemma}
\begin{proof}
Assume that the ring $\frac{A[M]}{(I \cap M)A[M]}$ is subintegrally closed in $\frac{B[N]}{IB[N]}$. Since $N$ is affine, Theorem~\ref{MainTheorem-1}(a) yields an isomorphism $\mathcal{I}(A, B) \cong \mathcal{I}\left(\frac{A[M]}{(I \cap M)A[M]}, \frac{B[N]}{IB[N]}\right).$
 Applying Theorem~\ref{MainTheorem-1}(d)(i), it follows that $A$ is subintegrally closed in $B$.
 We claim that $\Nil(A)=\Nil(B)$. Since $A\subset B$, we always have $\Nil(A)\subset\Nil(B)$. For the other direction, suppose $x \in \operatorname{Nil}(B), x \neq 0$. Let $r>1$ be the smallest integer such that $x^{r}=0$. Then $x^{i}=0 \in A$ for all $i \geq r$. Thus $\left(x^{r-1}\right)^{2} \in A$ and $\left(x^{r-1}\right)^{3} \in A$. Since $A$ is subintegrally closed in $B$, we get $x^{r-1} \in A$. The decreasing induction can be continued until we get $x \in A$, which proves our claim $\Nil(A)=\Nil(B)$.
\par
    Next we claim that $A/\Nil(A)$ is subintegrally closed in $B/\Nil(B)$. To see this, suppose $\overline{b} = b + \Nil(B) \in \overline{B}$ be such that $\overline{b}^2, \overline{b}^3 \in \overline{A}$. Then we have
\[
\overline{b}^2 = b^2 + \Nil(B) \in \overline{A},
\]
so there exists $u \in A$ such that $b^2 + \Nil(B) = u + \Nil(A)$. Since $\Nil(A) = \Nil(B)$ in our setting, it follows that
\[
b^2 + \Nil(A) = u + \Nil(A),
\]
and hence $b^2 \in A$. Similarly, we get $b^3 \in A$. As $A$ is subintegrally closed in $B$, it follows that $b \in A$. Consequently,
\[
\overline{b} = b + \Nil(B) = b + \Nil(A) \in \overline{A}.
\]
This proves that $\overline{A}$ is subintegrally closed in $\overline{B}$.

\par
  Now, since $I$ is a radical ideal, by \cite[Proposition 2.36]{B-G} we can assume that $I=\p_1\cap \ldots \cap \p_n$, where $\mathfrak{p_i}$'s are prime ideals in $N$. We prove the result by induction on $n$. If $n=1$, then $I=\mathfrak{p}_{1}$. Let $N^{\prime}=N \backslash \mathfrak{p}_{1}$, and $M^{\prime}=M \backslash\left(\mathfrak{p}_{1} \cap M\right)$. Then $N^{\prime}$ and $M^{\prime}$ are commutative cancellative torsion-free submonoid of $N$ and $M$ respectively. 
 Observe that $\frac{A[M]}{(I \cap M) A[M]} \cong A\left[M^{\prime}\right]$ because both sides are the free $A$-module with basis those elements $m \in M$ that are not contained in $I \cap M$. Similarly we have $B\left[N^{\prime}\right] \cong \frac{B[N]}{I B[N]}$.
    Since by assumption  $\frac{A[M]}{(I\cap M)A[M]}$ is subintegrally closed in $\frac{B[N]}{IB[N]}$, in other words $A\left[M^{\prime}\right]$ is subintegrally closed in $B\left[N^{\prime}\right]$. Since $M' \subset N'$ is an extension of commutative cancellative torsion-free monoids, by \cite[see after Question~1.1]{Sarwar} it follows that $M'$ is subintegrally closed in $N'$.
Moreover, one can observe that
\[
\frac{\overline{A}[M]}{(I \cap M) \overline{A}[M]} \cong \overline{A}[M'] \quad \text{and} \quad \frac{\overline{B}[N]}{I \overline{B}[N]} \cong \overline{B}[N'].
\]
Since $\overline{A} \subset \overline{B}$ is an extension of reduced rings, with $\overline{A}$ subintegrally closed in $\overline{B}$ and $M'$ subintegrally closed in $N'$, it follows from \cite[Theorem~4.79]{B-G} that $\overline{A}[M']$ is subintegrally closed in $\overline{B}[N']$.
Therefore, $\frac{\overline{A}[M]}{(I \cap M)\overline{A}[M]} \cong \overline{A}[M']$ is subintegrally closed in $\overline{B}[N'] \cong \frac{\overline{B}[N]}{I \overline{B}[N]}$. This completes the proof in the case $n = 1$.
\par
From here onward the proof is essentially similar as in Lemma~\ref{Lm2}; however, we provide the full argument for the shake of completeness.
Assume that $n>1$. Let $J=\mathfrak{p}_{2} \cap \ldots \cap \mathfrak{p}_{n}$. Note that $J \overline{B}[N]+$ $\mathfrak{p}_{1} \overline{B}[N]=\left(J \cup \mathfrak{p}_{1}\right) \overline{B}[N]$. Similarly $(J \cap M) \overline{A}[M]+\left(\mathfrak{p}_{1} \cap M\right) \overline{A}[M]=$ $\left(\left(J \cup \mathfrak{p}_{1}\right) \cap M\right) \overline{A}[M]$. This leads to the following Milnor square in the category of ring extensions (as defined in Definition~\ref{Mil_Sqr_Ring_Ext}
 ):
\[
\xymatrix{
  \frac{\overline{A}[M]}{(I\cap M)\overline{A}[M]}\subset  \frac{\overline{B}[N]}{I\overline{B}[N]} \ar[r]^{p_1} \ar@<-2pt>[d]_{p_2} & \frac{\overline{A}[M]}{(J\cap M)\overline{A}[M]}\subset \frac{\overline{B}[N]}{J\overline{B}[N]} \ar@<-2pt>[d]_{q_1} \\
  \frac{\overline{A}[M]}{(\p_1\cap M)\overline{A}[M]} \subset  \frac{\overline{B}[ N]}{\p_1\overline{B}[N]} \ar[r]^(.38){q_2} & \frac{\overline{A}[M]}{((J\cup \p_1)\cap M)\overline{A}[M]} \subset  \frac{\overline{B}[N]}{(J\cup \p_1)\overline{B}[N]}}
\]
\par
This consists of two Milnor squares in the category of commutative rings, one to the left of the inclusions, and one to the right, with the inclusions $\subset$ giving a map between them. The left Milnor square is constructed as in Example~\ref{LmFibProd_1} 
with $A$ in Example~\ref{LmFibProd_1} 
replaced by $\overline{A}[M], I$ replaced by $(J \cap M) \overline{A}[M]$ and $J$ replaced by $\left(\mathfrak{p}_{1} \cap M\right) \overline{A}[M]$. The right Milnor square is similarly constructed from Example~\ref{LmFibProd_1} 
with $\overline{A}$ replaced by $\overline{B}[N], I$ replaced by $J \overline{B}[N]$ and $J$ replaced by $\mathfrak{p}_{1} \overline{B}[N]$. The upper left $\subset$ was given a name $i_{I}$, which is an inclusion $i_{I}: \frac{\overline{A}[M]}{(I \cap M) \overline{A}[M]} \hookrightarrow \frac{\overline{B}[N]}{I \overline{B}[N]}$. Similarly inclusions $i_{J}, i_{\mathfrak{p}_{1}}$, and $i_{J \cup P_{1}}$ are defined, each indicated by $\subset$ at the appropriate place in the diagram. At the coset level $i_{I}(a+(I \cap M) \overline{A}[M])=a+I \overline{B}[N]$ for some $a \in \overline{A}[M]$. The existence or not of such an $a$ may not be apparent if we are given a coset $b+I \overline{B}[N] \in \frac{\overline{B}[N]}{I \overline{B}[N]}$, and if $a$ exists it will not be unique. And as sets $(a+(I \cap M) \overline{A}[M])$ and its image $a+I \overline{B}[N]$ are different, the first being a copy of $(I \cap M) \overline{A}[M]$ and the second of $I \overline{B}[N]$. To simplify notation we may abusively write $f \in \frac{\overline{B}[N]}{I \overline{B}[N]}$ instead of $f+I \overline{B}[N], f \in \overline{B}[N]$. If $f \in \operatorname{Im}\left(i_{I}\right)$ it is convenient to write $f \in \frac{\overline{A}[M]}{(I \cap M) \overline{A}[M]}$ which will not mean that $f \in \overline{A}[M]$, but rather that the coset $f+I \overline{B}[N]$ contains an element of $\overline{A}[M]$. The restriction of $p_{1}$ to $\frac{\overline{A}[M]}{I \cap M) \overline{A}[M]}$ will be denoted by $p_{1 \mid}$. Similarly we write $p_{2 \mid}, q_{1 \mid}$, and $q_{2\mid}$.
\par
We wish to prove that $\frac{\overline{A}[M]}{(I \cap M) \overline{A}[M]}$ is subintegrally closed in $\frac{\overline{B}[N]}{I \overline{B}[N]}$. This will be the case if every element $f \in \frac{\overline{B}[N]}{I \overline{B} N}$ such that $f^{k} \in \frac{\overline{A}[M]}{(I \cap M) \overline{A}[M]}, k \geq 2$ is already in $\frac{\overline{A}[M]}{(I \cap M) \overline{A}[M]}$. So start with such an $f \in \overline{B}[N]$ (variously thought of as $f \in \frac{\overline{B}[N]}{I \overline{B}[N]}$ or as the coset $f+I \overline{B}[N]$). Then we have $p_{1}(f)=f+J \overline{B}[N]$ and for $k \geq 2$ we have $p_{1}\left(f^{k}\right)=$ $p_{1}(f)^{k} \in \frac{\overline{A}[M]}{(J \cap M) \overline{A}[M]}$. By the induction hypothesis $\frac{\overline{A}[M]}{(J \cap M) \overline{A}[M]}$ is subintegrally closed in $\frac{\overline{B}[N]}{J \overline{B}[N]}$ so $p_{1}(f) \in \frac{\overline{A}[M]}{(J \cap M) \overline{A}[M]}$. Similarly $\frac{\overline{A}[M]}{\left(\left(\mathfrak{p}_{1}\right) \cap M\right) \overline{A}[M]}$ is subintegrally closed in $\frac{\overline{B}[N]}{\mathfrak{p}_{1} \overline{B}[N]}$ so $p_{2}(f) \in \frac{\overline{A}[M]}{\left(\mathfrak{p}_{1} \cap M\right) \overline{A}[M]}$. Therefore $p_{1}(f)=$ $f+J \overline{B}[N]=a_{1}+J \overline{B}[N]$ and $p_{2}(f)=f+\mathfrak{p}_{1} \overline{B}[N]=a_{2}+\mathfrak{p}_{1} \overline{B}[N]$, for some $a_{1}, a_{2} \in \overline{A}[M]$. These must have the same image when mapped into the lower right corner of the right hand Milnor square so we must have $a_{1}+\left(J \cup p_{1}\right) \overline{B}[N]=a_{2}+\left(J \cup \mathfrak{p}_{1}\right) \overline{B}[N]$. Now we have $a_{1}-a_{2} \in \overline{A}[M]$ and $a_{1}-a_{2} \in\left(J \cup \mathfrak{p}_{1}\right) \overline{B}[N]$. But $\overline{A}[M] \cap\left(J \cup \mathfrak{p}_{1}\right) \overline{B}[N]=\left(\left(J \cup \mathfrak{p}_{1}\right) \cap M\right) \overline{A}[M]$ so $a_{1}+\left(\left(J \cup \mathfrak{p}_{1}\right) \cap M\right) \overline{A}[M]=a_{2}+\left(\left(J \cup \mathfrak{p}_{1}\right) \cap M\right) \overline{A}[M]$. This means that $a_{1}+(J \cap M) \overline{A}[M] \in \frac{\overline{A}[M]}{(J \cap M) \overline{A}[M]}$ and $a_{2}+\left(\mathfrak{p}_{1} \cap M\right) \overline{A}[M] \in \frac{\overline{A}[M]}{\left(\mathfrak{p}_{1} \cap M\right) \overline{A}[M]}$ have the same image in the lower right hand corner of the left Milnor square. Therefore they patch in the left hand Milnor square to $g=$ $a+(I \cap M) \overline{A}[M]$ such that $p_{1 \mid}(g)=a+(J \cap M) \overline{A}[M]=a_{1}+(J \cap M) \overline{A}[M]$ and $p_{2 \mid}(g)=a+\left(\mathfrak{p}_{1} \cap M\right) \overline{A}[M]=a_{2}+\left(\mathfrak{p}_{1} \cap M\right) \overline{A}[M]$. Now we use the inclusions to map this pull back situation to the right Milnor square. This yields $p_{1} i_{I}(g)=i_{J} p_{1\mid}(g)=i_{J}\left(a_{1}+(J \cap M) \overline{A}[M]\right)=a_{1}+J \overline{B}[N]=$ $p_{1}(f)$ and $p_{2} i_{I}(g)=i_{\mathfrak{p}_{1}} p_{2 \mid}(g)=i_{\mathfrak{p}_{1}}\left(a_{2}+\left(\mathfrak{p}_{1} \cap M\right) \overline{A}[M]\right)=a_{2}+\mathfrak{p}_{1} \overline{B}[N]=$ $p_{2}(f)$. By the unique pullback property of the right Milnor square we have $i_{I}(g)=f$ which implies that $f \in \frac{\overline{A}[M]}{(I \cap M) \overline{A}[M]}$ completing the proof of the lemma.
\end{proof}
%%%%%%%%%%%%%%%%%%%%%%%%%%%%%%%%%%%%%%%%%%%%%%%%%%%%%%%%%%
\vspace{.3cm}
\par
 \textbf{Proof of (d)(iii):} Let us look at the commutative diagram 
  \[
 \xymatrix{
  \CI (A, B)\ar@{->}[r]^(.36){\phi_{1}}\ar@{->}[d]^{\phi_{2}} & \CI (A/\Nil(A),  B/\Nil(B)) \ar@{->}[d]^{\phi_{3}}\\
   \CI \left(\frac{A[M]}{(I\cap M)A[M]}, \frac{B[N]}{IB[N]}\right)\ar@{->}[r]^{\phi_{4}} & \CI \left(\frac{\bar{A}[M]}{(I\cap M)\bar{A}[M]},\frac{\bar{B}[N]}{I\bar{B}[N]}\right)
   }
\]
 where $\phi_{i}$ are natural maps, and $\bar{A}=A/\Nil(A) $, $\bar{B}=B/\Nil(B)$.
By assumption, we are given that
$\phi_2: \CI(A, B) \cong \CI\left( \frac{A[M]}{(I \cap M)A[M]}, \frac{B[N]}{IB[N]} \right)$,
and since $N$ is affine, from part~(d)(ii), we get $\frac{A[M]}{(I\cap M)A[M]}$ is subintegrally closed in $\frac{B[N]}{IB[N]}$. By Lemma~\ref{Lm_Bar}, it follows that  $\frac{\overline{A}[M]}{(I \cap M)\overline{A}[M]}$ is subintegrally closed in $\frac{\overline{B}[N]}{I\overline{B}[N]}$.
Therefore, by applying part~(a), we conclude that the map
\[
\phi_3: \CI(\overline{A}, \overline{B}) \longrightarrow \CI\left( \frac{\overline{A}[M]}{(I \cap M)\overline{A}[M]}, \frac{\overline{B}[N]}{I\overline{B}[N]} \right)
\]
is an isomorphism.
As we have seen in the proof of Lemma~\ref{Lm_Bar} $\Nil(A)=\Nil(B)$, by \cite[Proposition 2.6]{Roberts and Singh} we get that $\phi_{1}$ is an isomorphism.
Therefore, $\phi_{2}$ is an isomorphism if and only if $\phi_{4}$ is an isomorphism.
By \cite[Proposition 2.7]{Roberts and Singh},  $\phi_{4}$ is an isomorphism if and only if
 \[\ker(\phi_{4}) = \frac{1+\Nil\left(\frac{B[N]}{IB[N]}\right)}{1+\Nil\left(\frac{A[M]}{(I\cap M)A[M]}\right)}.\]
By Theorem \ref{Nil-Thm}, we get 

 \[\frac{1+\Nil\left(\frac{B[N]}{IB[N]}\right)}{1+\Nil\left(\frac{A[M]}{(I\cap M)A[M]}\right)} = \frac{1+\frac{\Nil(B)[N]}{I\Nil(B)[N]}}{1+\frac{\Nil(B)[M]}{(I\cap M)\Nil(B)[M]}}\]
and by  Lemma \ref{Ker-Main Theorem} 
 
 \[\frac{1+\frac{\Nil(B)[N]}{I\Nil(B)[N]}}{1+\frac{\Nil(B)[M]}{(I\cap M)\Nil(B)[M]}}\]
is a trivial group if and only if $B$ is reduced or $M=N$.
% \[\Nil(B)= 0 \quad or \quad \frac{\Nil(B)[M]}{(I\cap M)\Nil(B)[M]}  \cong \frac{\Nil(B)[N]}{I\Nil(B)[N]}.\] 
Thus $\phi_4$ is an isomorphism if and only if $B$ is reduced or $M=N$.
% \[\Nil(B)= 0\quad or \quad \frac{\Nil(B)[M]}{(I\cap M)\Nil(B)[M]}  \cong \frac{\Nil(B)[N]}{I\Nil(B)[N]};\]
Hence $\phi_2$ is an isomorphism implies either $B$ is reduced or $M=N$.
 $\hfill \gj$
\end{proof}

\begin{corollary}
 Let $A\subset B$ be an extension of rings, and $I$ an ideal in $B[X_1,\ldots,X_n]$ generated
 by monomials such that $I=\sqrt{I}$. Then $A$ is subintegrally closed in $B$ if and only if 
 $\CI(A,B)\cong \CI(A[X_1,\ldots,X_n]/I,B[X_1,\ldots,X_n]/I)$.
\end{corollary}
\begin{proof}
    Follows from part (c) and (d) (i) of Theorem \ref{MainTheorem-1} by taking  $M=N=\mathbb{N}^{n}$.
    $\hfill \gj$
\end{proof}

\vspace{.5cm}

\subsection{A result on non-subintegrally closed extension of monoid algebras}

Let $A\subset B$ be an extension of rings and let $M$ be a positive monoid. Also let $^+\!\!A$ be the subintegral closure of $A$ in $B$. The canonical map $\theta(A,B):\CI(A,B)\longrightarrow \CI(A[M],B[M])$ is given by $\theta(J)=JA[M]$ and is a group homomorphism. Corresponding to the inclusions  $A\subset ^+\!\!A\subset B$, we get the inclusion map $i: \CI(A,^+\!\!A)\longrightarrow \CI(A,B)$ given by $J\longmapsto J$ and the map $\phi(A,^+\!\!A,B) : \CI(A,B)\longrightarrow \CI(^+\!\!A,B)$ is given by $J\longmapsto J^+\!\!A$. For a radical ideal $I$ in $M$ and looking at bases we have inclusions of rings
$$
\frac{A[M]}{I A[M]} \subset \frac{{ }^{+} A[M]}{I^{+} A[M]} \subset \frac{B[M]}{I B[M]}.
$$
Thus we have maps in the bottom row of the diagram in the next theorem. 
 
 \begin{theorem}\label{diagram-thm}
 Let $A\subset B$ be an extension of rings and let $^+\!\!A$ denote the subintegral closure of $A$ in $B$. Assume $M$ is an affine positive monoid, and $I$ is a radical ideal in $M$. Then\\
 (i) the following diagram 
 
 \[
% {\large
\begin{tikzcd}
1 \arrow[r] & \CI(A,^+\!\!A) \arrow[r, "i"] \arrow[d, "{\theta(A,^+\!\!A)}"'] & \CI(A,B) \arrow[r, "{\phi(A,^+\!\!A,B)}"] \arrow[d, "{\theta(A,B)}"'] & \CI(^+\!\!A,B) \arrow[r] \arrow[d, "{       \theta(^+\!\!A,B)}" '] & 1 \\
1 \arrow[r] & \CI\left(\frac{A[M]}{IA[M]},\frac{^+\!\!A[M]}{I^+\!\!A[M]}\right) \arrow[r]                                    & \CI\left(\frac{A[M]}{IA[M]},\frac{B[M]}{IB[M]}\right) \arrow[r, "\phi'"']                                      & \CI\left(\frac{^+\!\!A[M]}{I^+\!\!A[M]},\frac{B[M]}{IB[M]}\right) \arrow[r]                               & 1
\end{tikzcd}
%}
\]
is commutative with exact rows, where $\phi' = \phi\left(\frac{A[M]}{IA[M]},\frac{^+\!\!A[M]}{I^+\!\!A[M]},\frac{B[M]}{IB[M]}\right)$ is defined above.\\
(ii) If $\mathbb{Q}\subset A$, then $\CI\left(\frac{A[M]}{IA[M]},\frac{^+\!\!A[M]}{I^+\!\!A[M]}\right)\cong \frac{\mathbb{Z}[M]}{I\mathbb{Z}[M]}\otimes_{\mathbb{Z}}\CI(A,^+\!\!A).$
 \end{theorem}
 
\begin{proof}
%$(i)$ The arguments of the proof are the same as in the proof of \cite[Theorem 1.4(i)]{Sarwar}.
$(i)$ Since the maps are natural, the diagram is commutative. 
By \cite[Section 3]{Singh} the rows are exact except at the right. Now we have only to prove that the maps $\phi(A,^+\!\!A,B)$ and $\phi'$ are surjective. Since $^+A$ is subintegrally closed in $B$, the map $\theta(^+\!\!A,B)$ is an isomorphism by Theorem \ref{MainTheorem-1}(c). By using commutativity of the diagram, it is enough to show that $\phi(A,^+\!\!A,B)$ is surjective. But this follows from \cite[Proposition 4.1] {Sadhu2015} by taking $C=^+\!\!A$ (also see \cite{Isc89} for similar results in Picard groups).
\par
$(ii)$ In  \cite{Roberts and Singh} a natural map $\xi_{B / A}: B / A \rightarrow \mathcal{I}(A, B)$ is defined, and is an isomorphism if $A \subset B$ is a subintegral extension of $\mathbb{Q}$-algebras (see  \cite[Main Theorem 5.6]{Roberts and Singh} and \cite[Theorem 2.3]{Reid Roberts Singh}).
 Hence by \cite[Lemma 5.3]{Roberts and Singh}, we have a commutative diagram
\[
%{\large
\begin{tikzcd}
\frac{^+\!\!A}{A} \arrow[r, "\xi_{^+\!\!A/A}"] \arrow[d, "j"'] & \CI(A,^+\!\!A) \arrow[d, "{\theta(A,^+\!\!A)}"] \\
\frac{^+\!\!A[M]}{I^+\!\!A[M]}/\frac{A[M]}{IA[M]} \arrow[r, "\xi"]                         & \CI\left(\frac{A[M]}{IA[M]},\frac{^+\!\!A[M]}{I^+\!\!A[M]}\right),                             
\end{tikzcd}
%}
\]
where $\xi:= \xi_{\frac{^+\!\!A[M]}{I^+\!\!A[M]}/\frac{A[M]}{IA[M]}}$. Since $A\subset ^+\!\!\!A$ is a subintegral extension, by Lemma \ref{Lm3}, $\frac{A[M]}{IA[M]}\subset \frac{^+\!\!A[M]}{I^+\!\!A[M]}$ is also a subintegral extension.
Thus $\xi_{^+\!\!A/A}$ and $\xi$ are both isomorphisms.
Therefore, we have
$\CI\left(\frac{A[M]}{IA[M]},\frac{^+\!\!A[M]}{I^+\!\!A[M]}\right)
\cong \frac{^+\!\!A[M]/I^+\!\!A[M]}{A[M]/IA[M]}$.

%$\CI(\frac{A[M]}{IA[M]},\frac{^+\!\!A[M]}{I^+\!\!A[M]})\cong\frac{^+\!\!A[M]/I^+\!\!A[M]}{A[M]/IA[M]}\cong\frac{^+\!\!A[M]/A[M]}{I^+\!\!A[M]/IA[M]} \cong \frac{^+\!\!A[M]/A[M]}{I\frac{^+\!\!A[M]}{A[M]}}\cong \frac{\frac{^+\!\!A}{A}[M]}{I(\frac{^+\!\!A}{A}[M])}\cong \frac{\mathbb{Z}[M]}{I\mathbb{Z}[M]}\otimes_{\mathbb{Z}}\frac{^+\!\!A}{A}\cong \frac{\mathbb{Z}[M]}{I\mathbb{Z}[M]}\otimes_{\mathbb{Z}}\CI(^+\!\!A,A)$.

Also we have 
  \[
 \begin{split}
\frac{\mathbb{Z}[M]}{I\mathbb{Z}[M]}\otimes_{\mathbb{Z}}\CI(^+\!\!A,A) &\cong \frac{\mathbb{Z}[M]}{I\mathbb{Z}[M]}\otimes_{\mathbb{Z}}\frac{^+\!\!A}{A}\\ 
  &\cong \frac{\frac{^+\!\!A}{A}[M]}{I(\frac{^+\!\!A}{A}[M])}  \\
   &\cong \frac{^+\!\!A[M]/A[M]}{I\frac{^+\!\!A[M]}{A[M]}} \\
   &\cong \frac{^+\!\!A[M]/A[M]}{I^+\!\!A[M]/IA[M]} \ \ (\text{By Remark} \ \ref{Iso}(1))\\
   &\cong \frac{^+\!\!A[M]/I^+\!\!A[M]}{A[M]/IA[M]}\  \
   (\text{By Remark} \ \ref{Iso}(2))\\
   &\cong \CI\left(\frac{A[M]}{IA[M]},\frac{^+\!\!A[M]}{I^+\!\!A[M]}\right).
 \end{split}
 \]
Hence, we get $\CI\left(\frac{A[M]}{IA[M]},\frac{^+\!\!A[M]}{I^+\!\!A[M]}\right)
\cong\frac{\mathbb{Z}[M]}{I\mathbb{Z}[M]}\otimes_{\mathbb{Z}}\CI(^+\!\!A,A)$.
$\hfill \gj$
\end{proof} 
\begin{remark}\label{Iso}
$(1)$  Note that the surjection $I^+\!\!A[M]
  \twoheadrightarrow{I\frac{^+\!\!A[M]}{A[M]}}$ with kernel $IA[M]$ induces an isomorphism $\frac{I^+\!\!A[M]}{IA[M]}\cong{I\frac{^+\!\!A[M]}{A[M]}}$.

  $(2)$ Note that $\frac{^+\!\!A[M]/A[M]}{I^+\!\!A[M]/IA[M]} \cong \frac{ ^+\!\!A[M]}{(A[M]+I^+\!\!A[M])}$
  and also $ \frac{^+\!\!A[M]/I^+\!\!A[M]}{A[M]/IA[M]} \cong \frac{ ^+\!\!A[M]}{(A[M]+I^+\!\!A[M])}.$ Both are isomorphic as an abelian group.
\end{remark}

\vspace{0.5cm}

\section{Acknowledgment} 
Md Abu Raihan acknowledges the financial support (file no. 09/081(1388)/2019-EMR-I) of the Council of Scientific and Industrial Research (C.S.I.R.), Government of India. H.~P. Sarwar would like to thank SRIC, IIT Kharagpur for the ISIRD project (Project code: AKA) and S.E.R.B, DST, India for the project SRG/2020/000272.

%{\small

%}

\textsc{Md Abu Raihan, Department of Mathematics, Indian Institute of Technology Kharagpur, Kharagpur-721302, West Bengal, India }\\
 Email address: \textbf{aburaihan908@gmail.com}\\

 \textsc{Leslie G. Roberts, Department of Mathematics and Statistics, Queen's University, Kingston, Ontario, Canada, K7L 3N6}\\
 Email address: \textbf{robertsl@queensu.ca}\\

\textsc{Husney Parvez Sarwar, Department of Mathematics, Indian Institute of Technology Kharagpur, Kharagpur-721302, West Bengal, India }\\
 Email address: \textbf{parvez@maths.iitkgp.ac.in}; \textbf{mathparvez@gmail.com}


\begin{thebibliography}{}
% \bibitem{An82}{} D.F. Anderson, {\it Seminormal graded rings II}, 
 %J. Pure Appl. Algebra {\bf 23} (1982), no. 3, 221-226.

 %\bibitem{AnA82} D.D. Anderson, D.F. Anderson, {\it Divisorial ideals and invertible 
 %ideals in a graded integral domain}, J. Algebra {\bf 76} (1982), no. 2, 549-569. 
 % \bibitem{JiriBook} Jiří Adámek, Horst  Herrlich, George E. Strecker, {\it Abstract and concrete categories. The joy of cats.} Reprint of the 1990 original (Wiley, New York). Repr. Theory Appl. Categ. No. 17 (2006), 1–507.  
 
 
 \bibitem{Alfosin}{} J.~L. Ram\'irez~Alfons\'in, {\it The Diophantine Frobenius problem}, Oxford Lecture Series in Mathematics and its Applications, 30, Oxford Univ. Press, Oxford, 2005.
 
  \bibitem{An78}{} D.~F. Anderson, {\it Projective modules over subrings of k[X,Y] generated by monomials}, Pacific J. Math. {\bf79} (1978), no. 1, 5--17.
 
  \bibitem{An88}{} D.~F. Anderson, {\it The Picard group of a monoid domain}, 
  J. Algebra {\bf 115} (1988), no. 2, 342-351. 
  
% \bibitem{An90} D.~F.  Anderson, {\it The Picard group of a monoid domain. II}, Arch. Math. (Basel) {\bf 55} (1990), no. 2, 143--145.
  
 \bibitem{BassBook} H. Bass, {\it Algebraic K-theory} W. A. Benjamin, Inc., New York-Amsterdam 1968 xx+762 pp.

  \bibitem{BoschBook} S. Bosch, {\it Algebraic geometry and commutative algebra} Universitext. Springer, London, 2013. x+504 pp.
 
 
 %\bibitem{BG02} W. Bruns and J. Gubeladze, {\it Polytopal linear retractions}, 
% Trans. Am. Math. Soc. {\bf 354} (2002) 179-203. 
 
 \bibitem{B-G}{} W. Bruns and J. Gubeladze, {\it Polytopes, Rings and K-Theory}, Springer Monographs in Mathematics, 2009.
  
% \bibitem{{CHWW2009}}  G. Cortiñas, C. Haesemeyer, Mark E. Walker and C. Weibel, {\it The K-theory of toric varieties}, Trans. Amer. Math. Soc. 361 (2009), no. 6, 3325–3341.
   
 \bibitem{CHMWW2012} G. Cortiñas, C. Haesemeyer, Mark E. Walker and C. Weibel, {\it The K-theory of toric varieties in positive characteristic}, J. Topol. {\bf 7} (2014), no. 1, 247--286. 
 
%\bibitem{CHMWW2015} G. Cortiñas, C. Haesemeyer, Mark E. Walker and C. Weibel, {\it Toric varieties, monoid schemes and cdh descent}, Journal fur die Reine und Angewandte Mathematik {\bf 698} (2015), 1--54. 

 \bibitem{CHMWW2018} G. Cortiñas, C. Haesemeyer, Mark E. Walker and C. Weibel,
 {\it The K-theory of toric schemes over regular rings of mixed characteristic}, Singularities, algebraic geometry, commutative algebra, and related topics, Springer, Cham, (2018),  455--479.
 
 \bibitem{CLS2011} D.~A. Cox, J.~B. Little and  H.~K. Schenck, {\it Toric varieties}, {\bf124}, American Mathematical Soc,  2011.

  \bibitem{Dayton} B. Dayton, {\it The Picard group of a reduced G-algebra}, J. Pure Appl. Algebra {\bf59} (1989) 237--253.
 
  
  \bibitem{Gu89} J. Gubeladze, {\it  The Anderson conjecture and a maximal class of monoids over which projective modules are free}, (Russian) Mat. Sb. (N.S.) {\bf135}(177) (1988), no. 2, 169–185, 271; translation in Math. USSR-Sb. {\bf63} (1989), no. 1, 165--180.
 
 \bibitem{Gu05} J. Gubeladze, {\it The nilpotence conjecture in K-theory of toric varieties}, Invent. Math. {\bf160} (2005), no. 1, 173--216.
 
 \bibitem{Gu18} J. Gubeladze, {\it  Unimodular rows over monoid rings}, Adv. Math. {\bf337} (2018), 193-–215.

\bibitem{Isc89} F. Ischebeck, {\it Subintegral Ring Extensions and some K-Theoretical Functors} J. Algebra, 121, 323--338 (1989).
 
  \bibitem{KeSa15} M.~K. Keshari and H.~P. Sarwar, {\it Serre dimension of monoid algebras}, Proc. Indian Acad. Sci. Math. Sci. {\bf127} (2017), no. 2, 269--280.
  
  \bibitem{KrSa17} A. Krishna and H.~P. Sarwar, {\it K-theory of monoid algebras and a question of Gubeladze}, J. Inst. Math. Jussieu {\bf 18} (2019), no. 5, 1051--1085.
  
  % \bibitem{Kr-Sa20} A. Krishna, and H.~P. Sarwar, {\it Negative K-theory and Chow group of monoid algebras}, K-theory in algebra, analysis and topology, 195–224, Contemp. Math., {\bf749}(2020), Amer. Math. Soc. 
   
  % \bibitem{MilnorBook} J. Milnor, {\it  Introduction to algebraic K-theory} Annals of Mathematics Studies, No. {\bf72}. Princeton University Press, Princeton, N.J.; University of Tokyo Press, Tokyo, 1971. xiii+184 pp.
   
%  \bibitem{SMandalBook} S. Mandal, {\it  Projective modules and complete intersections} Vol. 1672. Springer Science \& Business Media, 1997.
  
%  \bibitem{Reid and Roberts1996}{} L. Reid and Leslie G. Roberts, {\it A new criterion for weak subintegrality}, Communications in Algebra. 1996 Jan 1; {\bf24}(10):3335-3342.
  
  \bibitem{Reid Roberts Singh}{} L. Reid, Leslie G. Roberts and B. Singh, {\it Finiteness of subintegrality}, Algebraic K-theory and algebraic topology (Lake Louise, AB, 1991), 223-227, NATO Adv. Sci. Inst. Ser. C Math. Phys. Sci., {\bf407}, Kluwer Acad. Publ.,
  Dordrecht, 1993.
  
  \bibitem{Reid Roberts Singhweak}{}  L. Reid, Leslie G. Roberts and B. Singh, {\it On weak subintegrality}, Journal of Pure and Applied Algebra.  Dec 31; {\bf114}(1), (1996), 93--109.
  
  \bibitem{RR00} L. Reid and Leslie G. Roberts, {\it Monomial subrings and systems of subintegrality}, J. Pure Appl. Algebra {\bf151} (2000), no. 3, 287--299.
  
 \bibitem{RR01}  L. Reid and Leslie G. Roberts  {\it Monomial subrings in arbitrary dimension}, J. Algebra {\bf236} (2001), no. 2, 703--730. 
 

  \bibitem{Roberts and Singh} Leslie G. Roberts and  B. Singh, {\it Subintegrality, invertible modules and the Picard group}, Compositio Math. {\bf85}
(1993), no. 3, 249--279.

 \bibitem{Sadhu2015}{} V. Sadhu, {\it Subintegrality, invertible modules and Laurent polynomial extensions}, Proc. Indian Acad. Sci. Math. Sci. {\bf125} (2015), no. 2, 149-160.
  
 \bibitem{Sadhu and Singh} V. Sadhu and B. Singh, {\it Subintegrality, invertible modules and polynomial extensions}, J. Algebra {\bf393} (2013), 16--23.

  \bibitem{Sa16} H.~P. Sarwar, {\it Some results about projective modules over monoid algebras}, Comm. Algebra {\bf44} (2016), no. 5, 2256--2263. 
  
  \bibitem{Sarwar} H.~P. Sarwar, {\it Ideal class groups of monoid algebra},  Journal of Commutative Algebra, {\bf 9} (2017), no. 2, 303--312. 
  
  \bibitem{Sa21} H.~P. Sarwar, {\it $K_0$-stability over monoid algebras}, Ann. K-Theory {\bf6} (2021), no. 4, 629--649.

  \bibitem{Singh1996}{} B. Singh, {\it On subintegrality. Commutative algebra, algebraic geometry, and computational methods}, (Hanoi, 1996), 113–121, Springer, Singapore, 1999.

 \bibitem{Singh}{} B. Singh, {\it The Picard group and subintegrality in positive characteristic}, Compos. Math. {\bf95} (1995) 309-321. 
  
 
 \bibitem{Swan1980} R.~G. Swan, {\it On seminormality}, Journal of Algebra, {\bf67}(1), (1980), 210--229.

%\bibitem{Swan92}  R.~G. Swan, {\it Gubeladze proof of Anderson's conjecture}, Contemp. Math {\bf 124} (1992), 215--250.
  
% \bibitem{Vitulli2011}{} Marie A. Vitulli {\it Weak normality and seminormality}, Commutative Algebra (2011), 441--480.
 \bibitem{CW2011} Charles A. Weibel 
  {\it The $K$-book An Introduction to
Algebraic $K$-theory}, {\bf145}, Providence: American Mathematical Society,  2013.
 \bibitem{CW1994Hom} Charles A. Weibel
 {\it An introduction to homological algebra}, {\bf38}. Cambridge university press, 1994.
 
\bibitem{Yanagihara1985}{} H. Yanagihara, {\it On an intrinsic definition of weakly normal rings}, Kobe J. Math. {\bf2}, (1985), 89--98. 
 



 
\end{thebibliography}
 \end{document}